\pdfoutput=1 
\documentclass[11pt,a4paper,twoside, final]{article}

\usepackage[T1]{fontenc}
\usepackage[utf8]{inputenc}
\usepackage{lmodern}

\usepackage[english]{babel} 

\usepackage[a4paper,left=31mm,right=31mm,top=33mm,bottom=37mm]{geometry}

\usepackage{tikz}

\usepackage{graphicx}
\usepackage{caption}
\usepackage[labelformat=simple]{subcaption}

\usepackage[strict=true,autostyle,english=british]{csquotes}

\usepackage{authblk}

\usepackage[toc, page]{appendix}

\usepackage{amsmath,amsfonts,amssymb,amsthm}
\usepackage[ntheorem]{mathtools}                                  

\usepackage{ifthen}                    
\usepackage{etoolbox}                  
\usepackage{graphicx}                  
\usepackage{epsfig}                    
\usepackage{enumerate}                 
\usepackage{microtype}                 
\usepackage{dsfont}                    

\usepackage{xfrac}                     

\usepackage[colorlinks = true, linkcolor=black,
citecolor=black]{hyperref}

\newtheorem{theorem}{Theorem}[section]
\newtheorem{definition}[theorem]{Definition}
\newtheorem{lemma}[theorem]{Lemma}
\newtheorem{proposition}[theorem]{Proposition}
\newtheorem{corollary}[theorem]{Corollary}
\newtheorem{LemmaAndDefinition}[theorem]{Lemma and Definition}

\theoremstyle{remark}
\newtheorem{remark}[theorem]{Remark}

\numberwithin{equation}{section}

\newcommand{\R}{\mathbb{R}}
\newcommand{\N}{\mathbb{N}}
\newcommand{\C}{\mathbb{C}}
\newcommand{\Z}{\mathbb{Z}}

\newcommand{\eps}{\varepsilon}
\newcommand{\fhi}{\varphi}

\newcommand{\weak}{\rightharpoonup}

\newcommand{\twoscale}{\xrightharpoonup{\; 2 \;}}

\newcommand{\del}{\partial}

\renewcommand{\div}{\operatorname{div} \,}
\newcommand{\curl}{\operatorname{curl} \,}

\newcommand{\id}{\operatorname{id}}         
\newcommand{\eff}{\operatorname{eff}}

\renewcommand{\Re}{\operatorname{Re}}
\renewcommand{\Im}{\operatorname{Im}}
\renewcommand{\i}{\operatorname{i}}
\newcommand{\spann}{\operatorname{span} }
\newcommand{\diag}{\operatorname{diag}}

\newcommand{\indicator}[1]{\mathds{1}_{#1}}

\newcommand{\T}{\mathbb{T}}

\newcommand{\calL}{\mathcal{L}}

\newcommand{\calN}{\mathcal{N}}

\def\Xint#1{\mathchoice
   {\XXint\displaystyle\textstyle{#1}}%
   {\XXint\textstyle\scriptstyle{#1}}%
   {\XXint\scriptstyle\scriptscriptstyle{#1}}%
   {\XXint\scriptscriptstyle\scriptscriptstyle{#1}}%
   \!\int}
\def\XXint#1#2#3{{\setbox0=\hbox{$#1{#2#3}{\int}$}
     \vcenter{\hbox{$#2#3$}}\kern-.5\wd0}}
\newcommand{\meanint}{\Xint-}

\DeclarePairedDelimiterX\abs[1]\lvert\rvert{
\ifblank{#1}{\:\cdot\:}{#1}
}

\DeclarePairedDelimiterX\norm[1]\lVert\rVert{
\ifblank{#1}{\:\cdot\:}{#1}}

\DeclarePairedDelimiterX\scalarproduct[2]{\langle}{\rangle}{#1,#2}

\DeclarePairedDelimiterX\set[1]\lbrace\rbrace{\def\given{\colon}#1}                            

\newcommand{\Eeta}{E^{\eta}}
\newcommand{\Heta}{H^{\eta}}

\newcommand{\curly}{\operatorname{curl}_y \,}
\newcommand{\divy}{\operatorname{div}_y \,}
\renewcommand{\d}{\, \textrm{d}}

\newcommand{\e}{\textrm{e}}
\newcommand{\lineint}[1]{\int_{\gamma} #1}
\newcommand{\jump}[1]{[#1]}

\newcommand{\NE}{\mathcal{N}_{\Sigma}}
\newcommand{\LE}{\mathcal{L}_{\Sigma}}
\newcommand{\NH}{\mathcal{N}_{Y \setminus \overline{\Sigma}}}
\newcommand{\LH}{\mathcal{L}_{Y \setminus \overline{\Sigma}}}

\title{{\LARGE\bf Effective Maxwell's equations in general periodic
    microstructures
  }}

\author{ B.\,Schweizer, M.\,Urban\thanks{Technische Universit\"at
    Dortmund, Fakult\"at f\"ur Mathematik, Vogelpothsweg 87, D-44227
    Dortmund, Germany.}}

\date{March 16, 2017}

\begin{document}

\maketitle

\pagestyle{myheadings} 
\thispagestyle{plain} 

\markboth{B.\,Schweizer, M.\,Urban}{}

\begin{center}
\vskip4mm
\begin{minipage}[c]{0.86\textwidth}
  {\small {\bfseries Abstract:} We study the time harmonic Maxwell
    equations in a meta-material consisting of perfect conductors and
    void space. The meta-material is assumed to be periodic with
    period $\eta>0$; we study the behaviour of solutions $(\Eeta, \Heta)$
    in the limit $\eta\to 0$ and derive an effective system. In 
    geometries with a non-trivial topology, the limit system implies
     that certain components of the effective fields
    vanish. We identify the corresponding effective system and can predict,
    from topological properties of the meta-material, whether or not  it permits the
    propagation of waves.
    \\[2mm]

    }

\end{minipage}\\[3mm]
\end{center}

\section{Introduction}

We are interested in transmission properties of meta-materials. In
this context, a meta-material is a periodic assembly of perfect
conductors, and our question concerns the behaviour of electromagnetic
fields when the period of the meta-material tends to zero.  We fix a
frequency $\omega>0$ and investigate solutions to the time harmonic
Maxwell equations. Denoting the period of the meta-material by
$\eta>0$, we study the behaviour of solutions $(\Eeta, \Heta)$ to the
system
\begin{subequations}\label{eq:MaxSysSeq-intro}
  \begin{align}\label{eq:MaxSeq1-intro}
    \curl \Eeta &= \phantom{-}\i \omega \mu_0 \Heta
                  \quad\textrm{ in } \Omega\, ,\\
    \label{eq:MaxSeq2-intro}
    \curl \Heta &= - \i \omega \eps_0 \Eeta  \quad \:
                  \textrm{ in } \Omega \setminus \overline{ \Sigma}_{\eta}\, ,\\
    \label{eq:MaxSeq3-intro}
    \Eeta &= \Heta = 0 \qquad\   \textrm{ in } \Sigma_{\eta}\, ,
  \end{align}
\end{subequations}
in the limit $\eta\to 0$.  In this model, we assume that the perfect
conductor fills the subset $\Sigma_\eta\subset \Omega$ of the domain
$\Omega\subset\R^3$.

In general, meta-materials for Maxwell's equations are described with
two periodic material parameters $\eps$ and $\mu$ (permittivity and
permeability). We study here perfectly conducting inclusions, which
formally amount to setting $\eps = \infty$. In this case, the electric
and the magnetic field vanish in the inclusions; see \eqref
{eq:MaxSeq3-intro}. The material parameters in the other equations are
given by $\eps_0>0$ and $\mu_0>0$, the coefficients of vacuum.
Imposing \eqref {eq:MaxSeq1-intro} encodes boundary conditions: the
magnetic field $H$ has a vanishing normal component and the electric
field $E$ has vanishing tangential components on the boundary
$\del\Sigma_\eta$.

We ask: Can electromagnetic waves propagate in the periodic medium?
Are there components of the effective fields that necessarily vanish?
What is the effective system that describes the remaining components?

Our theory yields the following results as particular applications: In
a geometry with perfect conducting plates, transmission through the
meta-material is possible in two directions. Instead, in a geometry
with long and thin holes in the metal no transmission is possible.  

\subsection{Geometry and assumptions}\label{section:geometry and assumptions}
We are interested in studying general geometries
$\Sigma_\eta$. Nevertheless, we remain in the framework of standard
periodic homogenization, i.e., the set $\Sigma_\eta$ of inclusions is
locally periodic. A microscopic structure is considered, which is
given by a perfectly conducting part $\Sigma \subset Y$ in a single
periodicity cell $Y$, where $Y \coloneqq [-1/2, 1/2]^3$. 
We assume that the set $\Sigma$ is non empty and  open with a Lipschitz
boundary as a subset of the $3$-torus. 

Our aim is to study electromagnetic waves in an open subset
$\Omega \subset \R^3$. The meta-material is located in a second domain
$R \subset \subset \Omega$. In $\Omega \setminus R$, we have relative
permeability and relative permittivity equal to unity.
The microscopic structure in $R$ is defined using indices $k \in \Z^3$
and shifted cubes $Y_k^{\eta} \coloneqq \eta (k + Y)$, where
$\eta > 0$. By $\mathcal{K}$ we denote the index set
$\mathcal{K} \coloneqq \set{ k \in \Z^3 \given Y_k^{\eta} \subset R}$. We define the
meta-material $\Sigma_{\eta}$ by
\begin{equation}\label{eq:definition of Sigma eta}
  \Sigma_{\eta} \coloneqq \bigcup_{k \in \mathcal{K}} \eta (k +
  \Sigma) \subset R\, . 
\end{equation}

Even in the above periodic framework, quite general geometries and
topologies can be generated.  The simplest non-trivial example occurs
if we study a cylinder $\Sigma \subset Y$ that connects two opposite
faces of $Y$; see Figure \ref {fig:the metal cylinder}. The cylinder $\Sigma$
generates a set $\Sigma_\eta$ that is the union of disjoint long and  thin
fibers. In a similar way, we can generate the macroscopic geometry of
large metallic plates for which length and width are of macroscopic
size and the thickness is of order $\eta$; for the corresponding local
geometry $\Sigma$ see Figure \ref {fig:the metal plate}.

We investigate distributional solutions
$(\Eeta, \Heta) \in H^1(\Omega; \C^3) \times H^1(\Omega; \C^3)$
to~\eqref{eq:MaxSysSeq-intro}.  The number $\omega > 0$ denotes the
frequency, and $\mu_0, \eps_0 > 0$ are the permeability and the
permittivity in vacuum, respectively.  
We assume that we are given a sequence $(\Eeta, \Heta)_{\eta}$
of solutions to~\eqref{eq:MaxSysSeq-intro} that satisfies the
energy-bound
\begin{equation}
  \label{eq:energy-bound}
  \sup_{\eta > 0}\int_{\Omega}  \Big(\abs{\Heta}^2 + \abs{\Eeta}^2 \Big)< \infty\, .
\end{equation}
If~\eqref{eq:energy-bound} holds, by reflexivity of
$L^2(\Omega; \C^3)$, we find two vector fields
$E, H \in L^2(\Omega; \C^3)$ and subsequences such that
$\Eeta \weak E$ in $L^2(\Omega ; \C^3)$ and $\Heta \weak H$ in
$L^2(\Omega; \C^3)$.  Due to the compactness with respect to
two-scale convergence, we may additionally assume for fields
$E_{0}, H_{0} \in L^2(\Omega \times Y; \C^3)$ the two-scale convergence
\begin{equation}\label{eq:(Eeta, Heta) converges in two scale to (E 0,
    H 0)}
  \Eeta \twoscale E_0 \quad \textrm{ and } \quad \Heta \twoscale H_0\, .
\end{equation}

\subsection{Main results}
We obtain an effective system of equations that describes the limits
$E$ and $H$. The effective system depends on topological properties of
the microscopic geometry $\Sigma \subset Y$. We denote by
$\NE\subset \set{1,2,3}$ those directions for which no curve in
$\Sigma$ exists that connects corresponding opposite faces of $Y$
(the notation $\NE$ indicates that there is \enquote{no loop in $\Sigma$}; for a
precise definition see~\eqref{eq:index sets for H}). One of our result
is that the number of non-trivial components of the effective electric
field is given by $\abs{\NE}$ (see~Proposition \ref {proposition:
  dimension of solution space to the cell problem of E}). Similarly,
the number of non-trivial components of the effective magnetic field
is given by $\abs{\LH}$, where $\LH\subset \set{1,2,3}$ are those
directions for which a curve in $Y\setminus \overline{\Sigma}$ exists
that connects corresponding opposite faces of $Y$ (the notation $\LH$ indicates
that there is a \enquote{loop in $Y \setminus \overline{\Sigma}$}; for
the result see Proposition \ref {proposition:for every element of the
  set L E there is a solution to the cell problem}).

With the two index sets $\NE, \LH \subset \{1,2,3\}$, we can formulate
the effective system of Maxwell's equations. In the meta-material,
there holds
\begin{subequations}\label{eq:effective system - introduction}
  \begin{align}
    \curl \hat{E}   &=\phantom{-} \i \omega \mu_0 \hat{\mu} \hat{H}\,,\\
    (\curl \hat{H})_k &= -\i \omega \eps_0 (\hat{\eps} \hat{E})_k \qquad
                        \textrm{ for every } k \in \NE, \\
    \hat{E}_k     &= 0 \qquad \qquad \qquad \: \:\: \: \textrm{ for every }
                    k \in \set{1,2,3} \setminus \NE\, , \\
    \hat{H}_k &= 0 \qquad  \qquad \qquad \: \:\: \:  \textrm{ for every }
                k \in \set{1,2,3} \setminus \LH \, .
  \end{align}
\end{subequations}
In this set of equations, the effective relative permittivity
$\hat{\eps}$ and the effective relative permeability $\hat{\mu}$ are
defined through cell-problems. Our main result is the derivation of
these effective equations; see Theorem~\ref{theorem:macroscopic
  equations} below.  Essentially, the theorem states the following:
Let $(\Eeta, \Heta)_{\eta}$ be a bounded sequence of solutions
to~\eqref{eq:MaxSysSeq-intro} satisfying~\eqref{eq:energy-bound}, let
limit fields $(\hat{E}, \hat{H})$ be defined as weak and geometric
limits of $(\Eeta, \Heta)_{\eta}$, and let $\hat{\eps}$ and
$\hat{\mu}$ be the effective coefficients defined by cell-problems
(see~\eqref{eq:definition effective permittivity and permeability}). In
this situation, the limit $(\hat{E}, \hat{H})$ is a solution to the
effective system, which coincides with~\eqref{eq:effective system -
  introduction} in the meta-material.
Theorem~\ref{theorem:macroscopic equations} also specifies the
interface conditions along the boundary $\del R$ of the
meta-material.  The result allows to
determine, by checking topological properties of $\Sigma$,
if the meta-material supports propagating waves. To give an example:
In the case $\NE = \emptyset$ (that is, $\Sigma$ connects all opposite
faces of $Y$), the electric field $\hat E$ necessarily vanishes
identically in the meta-material and waves cannot propagate.

Of particular interest are those cases in which some components of
$\hat E$ and/or $\hat H$ vanish while the other components satisfy
certain blocks of Maxwell's equations. This occurs, e.g., in wire and
in plate structures. Our analysis is much more general: the effect
occurs when the solution spaces to the cell-problems are not three
dimensional, but have a lower dimension (drop of dimension).

\subsection{Literature}
From the perspective of applications, our contribution is closely
related to \cite {BouchitteSchweizer-Plasmons}, which is concerned
with an interesting experimental observation: Light can propagate well
in a structure made of thin silver plates; even nearly perfect
transmission through such a sub-wavelength structure was
experimentally observed. The mathematical analysis of \cite
{BouchitteSchweizer-Plasmons} explains the effect with a resonance
phenomenon. While \cite {BouchitteSchweizer-Plasmons} is purely
two-dimensional, the present contribution investigates which
genuinly three-dimensional structures are capable of showing similar
transmission properties.

From the perspective of methods, we follow other contributions more
closely.  We deal with the homogenization of Maxwell's equations in
periodic structures.  This mathematical task has already some history:
The book \cite {ZhikovMR1329546} contains the homogenization of the
equations in a standard setting (i.e., periodic and uniformly bounded
coefficient sequences $\eps_\eta$ and $\mu_\eta$); for this case see
also \cite {MR2029130}.

The first homogenization result for Maxwell's equations in a singular
periodic structure appeared in \cite {BouchitteSchweizer-Max}: Small
split rings with a large absolute value of $\eps_\eta$ were analyzed,
and a limit system with effective permittivity $\eps_0 \hat \eps$ and
effective permeability $\mu_0 \hat \mu$ was derived. The key point is
that the coefficient $\hat\mu$ of the limit system can have a negative
real part, due to resonance of the micro-structure. A similar result
was obtained in \cite {BouchitteBourel2009} with a simpler local
geometry; the effect of a negative $\Re \hat\mu$ is there obtained
through Mie-resonance. An extension to flat rings was performed in
\cite {Lamacz-Schweizer-Max}.  The construction of a negative index
material (with negative permittivity \emph{and} negative permeability)
was successfully achieved in \cite {Lamacz-Schweizer-2016} with the
additional inclusion of long and thin wires. For a recent overview we
refer to \cite {Schweizer-resonances-survey-2016}.

The step towards perfectly conducting materials was done in \cite
{Lipton-Schweizer}, in which~\eqref{eq:MaxSysSeq-intro} was also used.
The result of \cite {Lipton-Schweizer} is a limit system that takes
the usual form of Maxwell's equations, again with effective
permittivity $\eps_0 \hat \eps$, effective permeability
$\mu_0 \hat \mu$, and negative $\Re \hat\mu$.  Once more, the negative
coefficient is possible since the periodic structure $\Sigma_\eta$ has
a singular (torus-like) geometry.

Compared to the results described above, the work at hand takes a
different perspective: We are not interested in a negative
$\Re \hat\mu$, but we are interested to see whether or not certain
components of the effective fields have to vanish (due to geometrical
properties of the microstructure). If some components vanish, we want
to extract the equations for the remaining components.  The effect of
vanishing components is always a result of geometries in which the
substructure $\Sigma$ of the periodicity cell $Y$ touches two opposite
faces of $Y$. We recall that such substructures also enabled the
effect of a negative index in \cite {Lamacz-Schweizer-2016}. However,
in all contributions mentioned above (besides from \cite
{BouchitteSchweizer-Plasmons} and \cite {Lamacz-Schweizer-2016}) the
resonant structure $\Sigma$ is assumed to be compactly contained in
the cell $Y$.

It is worth mentioning that the study of wires (as a particular
example of a periodic microstructure with macroscopic dimensions) has
a longer history. Bouchitté and Felbacq showed that wire structures
with extreme coefficient values can lead to the effect of a negative
effective permittivity; see \cite{MR2262964, MR1444123}.  Related
wire-constructions have been analyzed by Chen and Lipton; see
\cite{ChenLipton2010, ChenLipton2013}.

Our results concern the scattering properties of periodic media. We
emphasize that, in contrast to many classical contributions, we treat
only the case that the period is small compared to the wave-length of
the incident wave (prescribed by the frequency $\omega$). Also in the
case that the period and the wave-length are of the same order, one
can observe interesting transmission properties, e.g., due to the
existence of guided modes in the periodic structure.  The
corresponding results are known as \enquote{diffraction theory} or
\enquote{grating theory}. For a fundamental analysis of existence and
uniqueness questions in such a diffraction problem we mention \cite
{MR1273315}; see also \cite {MR1961652}.  Regarding classical methods
we mention \cite {MR1160941}, where the transmission properties of a
periodically perturbed interface are studied by means of integral
methods. A more recent contribution regarding a similar periodic
interface is \cite {MR3335171}. Closer to our analysis is \cite
{MR1694448}, where a three-dimensional layer of a periodic material is
studied (the material is periodic in two directions); see also the
overview \cite {MR1335399}.

\subsection{Methods and organization}
We use the tool of two-scale convergence of~\cite{Allaire1992} and
consider the two-scale limits $E_0 = E_0(x, y)$ and $H_0=H_0(x, y)$ of
the fields $\Eeta$ and $\Heta$. By standard arguments, we obtain cell
problems for $E_0(x, \cdot)$ and $H_0(x, \cdot)$. We then characterise
the solution spaces of these cell problems---that is, we determine
bases of the solution spaces in terms of the index sets $\NE$ and
$\LH$. The crucial observation is the following: if the dimension of
one of the solution spaces is less than three, the standard procedure
to define homogenized coefficients does no longer work.  Hence the
form of the effective system is not clear.  However, once the
homogenized coefficients are carefully defined, the derivation of the effective
system is rather standard; see~\cite{Lamacz-Schweizer-2016,
  Lipton-Schweizer}.  Note that in \cite {BouchitteSchweizer-Max,
  Lamacz-Schweizer-Max, Lipton-Schweizer}, a full torus geometry was
considered; the fact that the complement of the torus is not simply
connected leads to a $4$-dimensional solution space in the cell
problem for $H$. We, however, are interested in the opposite effect:
Geometries that generate solution spaces of dimension smaller than
$3$.

In Section~\ref{sec:preliminary geometric results} we introduce the
notions of simple Helmholtz domains, $k$-loops, and the geometric
average, and prove auxiliary results. The derivation of the cell
problems and the characterisation of their solution spaces is carried
out in Section~\ref{sec:cell problems}. In Section~\ref{sec:derivation
  of the effective equations} we prove the main result, i.e., we
derive the effective system~\eqref{eq:effective system -
  introduction}. Section~\ref{section:examples} is devoted to the
discussion of some examples of microstructures.

\section{Preliminary geometric results}\label{sec:preliminary geometric results}

\subsection{Periodic functions, Helmholtz-domains, and $k$-loops}

Let $Y$ denote the closed cube $[-1/2, 1/2]^3$ in $\R^3$. We
define an equivalence relation $\backsim$ on $Y$ by identifying
opposite sides of the cube: $y_a \backsim y_b$ whenever
$y_a - y_b \in \Z^3$. The quotient space $Y / \backsim$ is
identified with the flat $3$-torus $\T^3$; the canonical projection
$Y \hookrightarrow \T^3$ is denoted by $\iota$.

A map $u \colon \R^3 \to \C^n$ is called \emph{$Y$-periodic} if
$u(\cdot + l) = u(\cdot)$ for all $l \in \Z^3$. For $m \in \N \cup
\set{\infty}$ and $n \in \N$, the function space $C_{\sharp}^m(Y ; \C^n)$
denotes the restriction of $Y$-periodic maps $\R^3 \to \C^n$ of class $C^m$ to $Y$; we may identify
this function space with $C^m(\T^3; \C^n)$.  Similarly, we define
\begin{equation*} 
  H_{\sharp}^m(Y ; \C^n) \coloneqq \set[\big]{u |_Y \colon Y \to \C^n \given u \in H^m_{\textrm{}loc}(\R^3; \C^n) \textrm{ is } Y \textrm{-periodic}} \, ,
\end{equation*}
which can be identified with $H^m(\T^3 ; \C^n)$.
We note that $L_{\sharp}^2(Y; \C^n) = H_{\sharp}^0(Y ; \C^n) = L^2(Y ;
\C^n)$. 

For a subset $U \subset Y$ such that $\iota(U) \subset \T^3$ is open,
we set $ H^m_{\sharp}(U ; \C^n)\coloneqq H^m(\iota(U);
\C^n)$. We denote by $H_{\sharp}(\curl, U)$ and $H_{\sharp}(\div, U)$
the spaces of all $L^2( \iota(U); \C^3)$-vector fields $u \colon U \to
\C^3$ such that the distributional curl and the distributional
divergence satisfy $\curl u \in L^2(\iota(U)); \C^3)$ and $\div u \in
L^2(\iota(U))$, respectively.  For brevity, we write $C_{\sharp}^k(U)$,
$L_{\sharp}^2(U)$,  and $H_{\sharp}^1(U)$ if no confusion about the
target space is possible. 

\begin{definition}[Simple Helmholtz domain]
 Let $U \subset Y$ be such that $\iota(U)$ is open in $\T^3$. We say
 that $U$ is a \emph{simple Helmholtz domain} if for every vector
 field $u \in L_{\sharp}^2(U; \C^3)$ with $\curl u = 0$ in $U$ there exist a potential $\Theta \in H_{\sharp}^1(U)$ and a constant $c_0 \in \C^3$ such that $u = \nabla \Theta + c_0$ in $U$.  
\end{definition}

\begin{remark}\label{remark:constant in Helmholtz domain is not unique}
  Note that, in general,  the constant $c_0$ is not unique. Take, for
  instance, $\Sigma \coloneqq B_r(0)$ with $r \in (0, 1/2)$. Then
  $\Theta(y) \coloneqq \lambda y_k$ and $c_0 \coloneqq - \lambda \e_k$
  yields a representation of $u =0$ for every $\lambda \in \C$ and $k
  \in \set{1,2,3}$. For a generalisation see Lemma \ref{lemma:existence of potentials in case of no k-loop}.
\end{remark}

In what follows, we consider curves $\gamma \colon [0, 1] \to Y$ (not
necessarily continuous in $Y$) such that
$\iota \circ \gamma \colon [0, 1] \to \T^3$ is continuous (see
Figure~\ref{fig:k-loop althoug Sigma not connected} for a subset
$U = \Sigma_1 \cup \Sigma_2 \subset Y$ that admits a discontinuous
path $\gamma$ in $U$ so that $\iota \circ \gamma$ is continuous). For
such a curve there exists a lift $\tilde{\gamma}$; that is, a
continuous curve $\tilde{\gamma} \colon [0, 1] \to \R^3$ with
$\pi \circ \tilde{\gamma} = \gamma$, where $\pi$ denotes the universal
covering $\R^3 \to \T^3$, $x \mapsto x \bmod \Z^3$.

\begin{definition}[$k$-loop]
 Let $U \subset Y$ be non empty and such that $\iota(U) \subset \T^3$ is open. Let $e_1, e_2,\e_3 \in \R^3$ be the standard basis
  vectors, and let $k \in \set{1,2,3}$. We say that a map $\gamma \colon [0, 1] \to
  \T^3$ is a \emph{$k$-loop in $\iota(U)$} if the path $ \gamma \colon [0, 1]
  \to \iota(U)$ is continuous and piecewise continuously differentiable, $\gamma(1) = \gamma(0)$, and $\scalarproduct{\tilde{\gamma}(1) - \tilde{\gamma}(0)}{\e_k} \neq 0$, where $\tilde{\gamma} \colon [0, 1] \to \R^3$ is a lift of $\gamma$.
   \end{definition}

\begin{remark}
  \textit{a)} For a lift $\tilde{\gamma}$ of the $k$-loop $\gamma$, we
  have that $\tilde{\gamma}(1) - \tilde{\gamma}(0) \in \Z^3$ by
  $\gamma(1) = \gamma(0)$.\\
  \textit{b)} By abuse of notation, we refer to
  $\gamma \colon [0,1] \to Y$ as a $k$-loop in $U$ when the curve
  $\iota \circ \gamma \colon [0, 1] \to U$ is a $k$-loop in $U$.
\end{remark}
For a subset $U$ of $Y$, we introduce the following index sets:
\begin{subequations} \label{eq:index sets for H}
\begin{align}
  \mathcal{L}_{U} &\coloneqq \set[\big]{k \in \set{1,2,3} \given \textrm{there is a $k$-loop in }  U} \label{eq:index sets for H 1}\, ,\\
  \mathcal{N}_{U} &\coloneqq \set[\big]{ k \in \set{1,2,3} \given \textrm{there is no $k$-loop in } U} \, . \label{eq:index sets for H 2}
\end{align}
\end{subequations}
Note that $\mathcal{L}_U \dot{\cup} \mathcal{N}_U = \set{1,2,3}$. We
turn to the analysis of potentials defined on $U$.
\begin{lemma}\label{lemma:existence of potentials in case of no k-loop}
  Let $U \subset Y$ be non-empty and such that $\iota(U) \subset \T^3$ is open. Assume further that $U$ has only finitely many connected components. Let $k$ be an element of $\set{1,2,3}$. If there is no $k$-loop in $U$, then there exists a potential $\Theta_k \in H_{\sharp}^1(U)$ such that $\nabla \Theta_k = \e_k$.
\end{lemma}

\begin{proof}
  We may assume that $\iota(U)$ is connected; otherwise each connected component is treated separately. We fix $k \in \set{1,2,3}$, and consider the two opposite faces 
  \begin{equation*}
    \Gamma_{k}^{(l)} \coloneqq \set{y \in Y \given y_k =-1/2} \quad \textrm{ and } \quad \Gamma_{k}^{(r)} \coloneqq \set{y \in Y \given y_k = +1/2}\, .
  \end{equation*}
  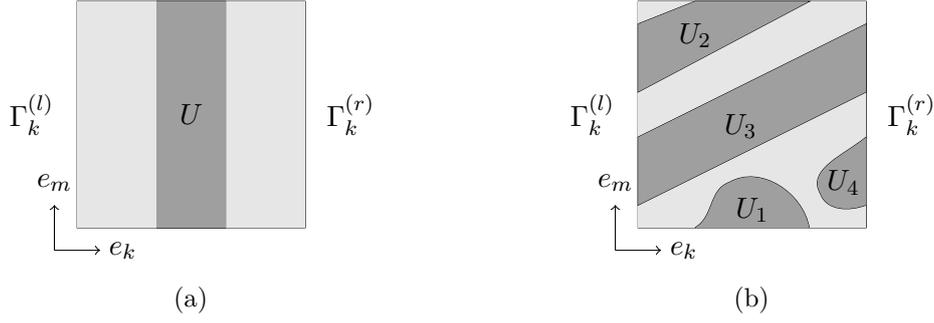
\begin{figure}
    \centering
    \begin{subfigure}[t]{0.49\textwidth}
      \hspace{1cm}
      \begin{tikzpicture}[scale=3]
  \draw[->] (-.6, -.6) -- (-.4, -.6);
  \draw[->] (-.6, -.6) -- (-.6, -.4);

  \coordinate (A) at (-0.5, -0.5);
  \coordinate (B) at (0.5,-0.5);
  \coordinate (C) at (0.5,0.5);
  \coordinate (D) at (-0.5, 0.5);
 
  \coordinate (E) at (-.15, -.5);
  \coordinate (F) at (-.15, .5);
  \coordinate (G) at (.15, .5);
  \coordinate (H) at (.15, -.5);
  \coordinate (I) at (.2, .2);

  \draw[black]  (A) -- (B) -- (C) -- (D) -- cycle;
  \fill[gray!20!white] (A) -- (B) -- (C) -- (D) -- cycle;
  
  \draw[black] (E) -- (F) -- (G) -- (H) -- cycle;
  \fill[gray!75!white] (E) -- (F) -- (G) -- (H) -- cycle;

  \node[] at (-.3, -.6) {$e_k$};
  \node[] at (-.6, -.3) {$e_m$};
  \node[] at (0, 0) {$U$};
        \node[] at (-.7, 0) {$\Gamma_k^{(l)}$};
        \node[] at (.7, 0) {$\Gamma_k^{(r)}$};
      \end{tikzpicture}
      \caption{ }
      \label{fig:proof of technical lemma - first case}
    \end{subfigure}
    \begin{subfigure}[t]{0.49\textwidth}
      \hspace{1cm}
      \begin{tikzpicture}[scale = 3]
  \draw[->] (-.6, -.6) -- (-.4, -.6);
  \draw[->] (-.6, -.6) -- (-.6, -.4);

  \coordinate (A) at (-0.5, -0.5);
  \coordinate (B) at (0.5,-0.5);
  \coordinate (C) at (0.5,0.5);
  \coordinate (D) at (-0.5, 0.5);
  
 \coordinate (E) at (-.5,-.4);
 \coordinate (F) at (.5,.1);
 \coordinate (G) at (.5, .4);
 \coordinate (H) at (-.5, -.1);
 \coordinate (I) at (-.5, .1);
 \coordinate (J) at (-.5, .1);
 \coordinate (K) at (.25, .5);
 \coordinate (L) at (-.25, .5);
 \coordinate (M) at (-.5, .4);
 \coordinate (N) at (-.25, -.5);
 \coordinate (N1) at (-.1, -.3);
 \coordinate (O) at (.25, -.5);
 \coordinate (P) at (.5, -.4);          
 \coordinate (P1) at (.3, -.35);
 \coordinate (Q) at (.5, -.1);

  \draw[black]  (A) -- (B) -- (C) -- (D) -- cycle;
  \fill[gray!20!white] (A) -- (B) -- (C) -- (D) -- cycle;
  
 \draw[black] (E)--(F) -- (G) -- (H) -- cycle;
 \fill[gray!75!white] (E)--(F) -- (G) -- (H) -- cycle;
 \draw[black] (J) --(K) -- (L) -- (M) -- cycle;
 \fill[gray!75!white] (J) --(K) -- (L) -- (M) -- cycle;
 \draw[black] (N) to [out=30, in=200] (N1)to[out=30, in=100] (O);
 \fill[gray!75!white] (N) to [out=30, in=200] (N1) to[out=30, in=100] (O);
 \draw[black] (P) to [out=200, in=290] (P1) to [out=120,in=220] (Q);
 \fill[gray!75!white](P) to [out=200, in=290] (P1) to [out=120,in=220] (Q);

  \node[] at (-.3, -.6) {$e_k$};
  \node[] at (-.6, -.3) {$e_m$};
  \node[] at (-.25, .35) {$U_2$};
        \node[] at (-.05, -.05) {$U_3$};
        \node[] at (.4, -.3) {$U_4$};
        \node[] at (0, -.42) {$U_1$};
        \node[] at (-.7, 0) {$\Gamma_k^{(l)}$};
        \node[] at (.7, 0) {$\Gamma_k^{(r)}$};
      \end{tikzpicture}
      \caption{ }
      \label{fig:proof of technical lemma - second case}
    \end{subfigure}
\caption{The sketches shows the projection of the periodicity cell $Y$
    to the $e_k$-$e_m$-plane. We assume that the geometry is
    independent of the third coordinate. (a) The set $U$ does not touch one of the
  boundaries $\Gamma_k^{(l)}$ or $\Gamma_k^{(r)}$, and hence $k \in
  \calN_U$. On the other hand, $m \in \calL_U$ since an $m$-loop in
  $U$ exists. (b) There are connected components of $U = U_1 \cup
  \dots \cup U_4$ that touch $\Gamma_k^{(l)}$ and $\Gamma_k^{(r)}$,
  but neither $k$-loops nor $m$-loops in $U$ exist. That is, $k, m
  \in \calN_U$.}
  \label{fig:1}
  \end{figure}

  \textit{Idea of the proof.}  By assumption, $U$ has only finitely
  many connected components $U_1, \dotsc, U_N$. Assume that none of
  the connected components touches the boundaries $\Gamma_k^{(l)}$ and
  $\Gamma_k^{(r)}$; that is, $U_i \cap \Gamma_k^{(l)} = \emptyset$ and
  $U_i \cap \Gamma_k^{(r)} = \emptyset$ for all
  $i \in \set{1, \dotsc, N}$ (as in Figure~\ref{fig:proof of technical
    lemma - first case}). In this case, we can define
  $\Theta_k \colon U \to \C$ by $\Theta_k(y) \coloneqq y_k$ and obtain
  a well-defined function $\Theta_k \in H_{\sharp}^1(U)$.

  Let us now assume that there are connected components
  $U_1, \dotsc, U_N$ such that $\iota(U_i \cup U_{i+1})$ is connected
  for all $i \in \set{1, \dotsc, N-1}$. The potential
  $\Theta_k \colon U_1 \cup \dots \cup U_N \to \C$ is defined as
  follows: On $U_1$, we set $\Theta_k(y) \coloneqq y_k$. On $U_2$, we
  define $\Theta_k(y) \coloneqq y_k + d_2$ for some $d_2 \in \Z$ such
  that $\Theta_k$ is continuous on $\iota(U_1 \cup U_2)$. This
  procedure can be continued until $\Theta_k$ is a continuous function
  on $U_1 \cup \dots \cup U_N$; the continuity of $\Theta_k$ is a
  consequence of the non-existence of a $k$-loop.

\textit{Rigorous proof.}  We denote by $\pi \colon \R^3 \to \T^3$, $x \mapsto x \bmod \Z^3$ the universal covering of $\T^3$. A lift $ [0, 1] \to \R^3$ of a continuous path $\gamma \colon [0, 1] \to \iota(U)$ is denoted by $\tilde{\gamma}$; that is, $\gamma = \pi \circ \tilde{\gamma}$. 

Fix a point $x_0 \in \iota(U)$. For every $x \in \iota(U)$ choose a
piecewise smooth path $\gamma \colon [0, 1] \to \iota(U)$ joining
$\gamma(0) = x_0$ and $\gamma(1) = x$. Let
$\tilde{\gamma} \colon [0,1] \to \R^3$ be a lift of $\gamma$, and
define $\Theta_k \colon \iota(U) \to \C$ by
$\Theta_k(x) \coloneqq \scalarproduct{\tilde{\gamma}(1) -
  \tilde{\gamma}(0)}{\e_k}$. Note that the non-existence of a $k$-loop
ensures that $\Theta_k$ is well defined. We further observe that the
difference $\tilde{\gamma}(1) - \tilde{\gamma}(0)$ is independent of
the chosen lift $\tilde{\gamma}$; indeed, for two lifts
$\tilde{\gamma}$ and $\tilde{\delta} $ of $\gamma$, there is
$l \in \Z^3$ such that $\tilde{\gamma} = \tilde{\delta} + l$. 
\end{proof}

\begin{corollary}\label{corollary:helmholtz decomposition with vanishing components of c_0}
  Let $U \subset Y$ be a simple Helmholtz domain. If $u \in L_{\sharp}^2(Y ;
  \C^3)$ satisfies $\curl u = 0$ in $U$, then there are a potential $\Theta \in H_{\sharp}^1(U)$ and a constant $c_0 \in \C^3$ such that $u = \nabla \Theta + c_0$ in $U$ and $\scalarproduct{c_0}{\e_k} = 0$ for every $k\in \mathcal{N}_U$.
\end{corollary}

\begin{proof}
  By the very definition of a Helmholtz domain, we find a potential $\Phi \in H_{\sharp}^1(U)$ and a constant $d_0 \in \C^3$ such that $u = \nabla \Phi + d_0$ in $U$. Due to Lemma~\ref{lemma:existence of potentials in case of no k-loop}, we find a potential $\Theta_k \in H_{\sharp}^1(U)$ such that $\nabla \Theta_k = \e_k$ for all $k \in \mathcal{N}_U$. Setting $\Theta \coloneqq \Phi + \sum_{k \in \mathcal{N}_U} \scalarproduct{d_0}{\e_k} \Theta_k$ provides us with an element of $H_{\sharp}^1(U)$. Moreover, 
  \begin{equation*}
    \nabla \Theta = u - d_0 + \sum_{k \in \mathcal{N}_U} \scalarproduct{d_0}{\e_k}\e_k = u - \sum_{k \in \mathcal{L}_U} \scalarproduct{d_0}{\e_k}\e_k\, .
  \end{equation*}
Setting $c_0 \coloneqq \sum_{k \in \mathcal{L}_U}
\scalarproduct{d_0}{\e_k}\e_k$, we find the desired representation of $u$.
\end{proof}

In the following, we need line-integrals of curl-free $L_{\sharp}^2(Y; \C^3)$-vector fields.
\begin{LemmaAndDefinition}[Line integral]\label{PropositonAndDefinition:line integral}
  Let $U \subset Y$ be such that $\iota(U)$ is an open subset of $\T^3$ with Lipschitz boundary. Assume that $\gamma \colon [0, 1] \to Y$ is such that $\iota \circ \gamma$ is a $k$-loop in $\iota(U)$. There exists a unique linear and continuous map
  \begin{equation*}
    I_{\gamma} \colon \set[\big]{u \in L_{\sharp}^2(Y; \C^3) \given \curl u = 0 \textrm{ in } U} \to \C \, , \quad u \mapsto I_{\gamma}(u)
  \end{equation*}
such that $I_{\gamma}(u)$ coincides with the usual line integral 
\begin{equation}\label{e:definition usual line integral}
  I_{\gamma}(u) = \lineint{u} \coloneqq \int_0^1 \scalarproduct[\big]{u\big(\gamma(t) \big)}{\dot{\gamma}(t)} \d t\, ,
\end{equation}
 for fields $u \in C_{\sharp}^1( Y; \C^3)$ that are curl free in $U$. 

The map $I_{\gamma}$ is called the \emph{line integral of $u$ along $\gamma$}, and we write, by abuse of notation, $\lineint{u}$ instead of $I_{\gamma}(u)$ for all $u \in L_{\sharp}^2(Y; \C^3)$ that are curl free in $U$. 
\end{LemmaAndDefinition}

\begin{proof}[Idea of a proof]
  The space $V \coloneqq \set{u \in C^1_{\sharp}(Y ; \C^3) \given \curl u = 0 \textrm{ in } U}$ is dense in $X \coloneqq \set{u \in L_{\sharp}^2(Y ; \C^3) \given \curl u = 0 \textrm{ in } U}$, which can be shown with a sequence of mollifiers $(\rho_{\eta})_{\eta}$. Defining $u_{\eta}$ by $u_{\eta} \coloneqq u * \rho_{\eta}$ provides us with a family $(u_{\eta})_{\eta}$ in $V$ with $u_{\eta} \to u$ in $L_{\sharp}^2(Y; \C^3)$. The map $\tilde{I}_{\gamma} \colon V \to \C$ defined by $\tilde{I}_{\gamma}(u) \coloneqq \lineint{u}$ is linear and continuous (because $u$ is curl free in $U$); using density of $V$ in $X$, the claim follows.
\end{proof}

We note the following: If $\gamma$ is a $k$-loop in $U$ and
$\tilde{\gamma}$ is one of its lifts, then
$\lineint{(u \circ \iota)} = \int_{\tilde{\gamma}} u$ for all fields
$u \in C_{\sharp}^1(U; \C^3)$. Indeed,
$\dot{\tilde{\gamma}} = \dot{\gamma}$, and
$u \circ \iota^{-1} \circ \gamma = u \circ \tilde{\gamma}$ for a
periodic function $u$. Note that the line integral along $\gamma$ and
the line integral along $\tilde{\gamma}$ coincide.

\subsection{The geometric average}
The notion of a geometric average was first introduced by Bouchitté, Bourel, and Felbacq in~\cite {BouchitteBourel2009}. The notion turned out to be very useful, it was extended in~\cite {Lipton-Schweizer} to more general geometries. Although we focus on simple Helmholtz domains in the main part of this work, we define the geometric average for general geometries.

In the subsequent definition of a geometric average, we need three
special curves $\gamma_1, \gamma_2, \gamma_3 \colon  [0, 1] \to Y$,
which represent paths along the edges---that is, $\gamma_1(t)
\coloneqq (t - 1/2, -1/2, -1/2)$, $\gamma_{2}(t) \coloneqq (-1/2,
t-1/2, -1/2)$, and $\gamma_3(t) \coloneqq (-1/2, -1/2, t-1/2)$. We
furthermore use the index set $\calL_U$ defined in~\eqref{eq:index sets for H 1}.
\begin{figure}%
  \centering
  \begin{subfigure}[b]{0.49\linewidth}
    \hspace{1cm}
  \begin{tikzpicture}[scale=3]
    \draw[->] (-.6, -.6) -- (-.4, -.6);
  \draw[->] (-.6, -.6) -- (-.6, -.4);

  \coordinate (A) at (-0.5, -0.5);
  \coordinate (B) at (0.5,-0.5);
  \coordinate (C) at (0.5,0.5);
  \coordinate (D) at (-0.5, 0.5);
  
 \coordinate (E) at (-.5, -.1);
 \coordinate (F) at (.1, .5);
 \coordinate (G) at (-.5, .1);
 \coordinate (H) at (-.1, .5);
 \coordinate (I) at (0.1, -.5);
 \coordinate (J) at (.5, -.1);
 \coordinate (K) at (.5, .1);
 \coordinate (L) at (-.1,-.5);

  \draw[black]  (A) -- (B) -- (C) -- (D) -- cycle;
  \fill[gray!20!white] (A) -- (B) -- (C) -- (D) -- cycle;
  
  \draw[black] (E) to [out=0, in=270] (F) to (H) to [out=270, in=0] (G);
  \fill[gray!75!white](E) to [out=0, in=270] (F) to (H) to [out=270, in=0] (G);
  \draw[black] (I) to [out=90, in=180] (J) to (K) to [out=180, in=90] (L);
  \fill[gray!75!white](I) to [out=90, in=180] (J) to (K) to [out=180, in=90] (L);

  \node[] at (-.3, -.6) {$e_k$};
  \node[] at (-.6, -.3) {$e_m$};
        \node[] at (-.7, 0) {$\Gamma_k^{(l)}$};
        \node[] at (.7, 0) {$\Gamma_k^{(r)}$};
        \node[] at (-.09, .2) {$U_1$};
        \node[] at (.07, -.25) {$U_2$};
  \end{tikzpicture}
   \caption{ }
   \label{fig:k-loop althoug Sigma not connected}
  \end{subfigure}
  \begin{subfigure}[b]{0.49\linewidth}
    \hspace{1.3cm}
    \begin{tikzpicture}[scale = 3]
  \draw[->] (-.6, -.6) -- (-.4, -.6);
  \draw[->] (-.6, -.6) -- (-.6, -.4);

  \coordinate (A) at (-0.5, -0.5);
  \coordinate (B) at (0.5,-0.5);
  \coordinate (C) at (0.5,0.5);
  \coordinate (D) at (-0.5, 0.5);
  
 \coordinate (E) at (-.5,-.4);
 \coordinate (F) at (.5,.1);
 \coordinate (G) at (.5, .4);
 \coordinate (H) at (-.5, -.1);
 \coordinate (I) at (-.5, .1);
 \coordinate (J) at (-.5, .1);
 \coordinate (K) at (.25, .5);
 \coordinate (L) at (-.25, .5);
 \coordinate (M) at (-.5, .4);
 \coordinate (N) at (-.25, -.5);
 \coordinate (O) at (.25, -.5);
 \coordinate (P) at (.5, -.4);          
 \coordinate (P1) at (.3, -.35);
 \coordinate (Q) at (.5, -.1);

  \draw[black]  (A) -- (B) -- (C) -- (D) -- cycle;
  \fill[gray!20!white] (A) -- (B) -- (C) -- (D) -- cycle;
  
 \draw[black] (E)--(F) -- (G) -- (H) -- cycle;
 \fill[gray!75!white] (E)--(F) -- (G) -- (H) -- cycle;
 \draw[black] (J) --(K) -- (L) -- (M) -- cycle;
 \fill[gray!75!white] (J) --(K) -- (L) -- (M) -- cycle;
  \draw[black] (N) to (Q);
 \fill[gray!75!white] (N) to (O);
  \draw[black] (P) -- (O) -- (N) -- (Q) -- cycle;
  \fill[gray!75!white] (P) -- (O) -- (N) -- (Q) -- cycle;

  \node[] at (-.3, -.6) {$e_k$};
  \node[] at (-.6, -.3) {$e_m$};
  \node[] at (-.25, .35) {$U_2$};
  \node[] at (-.09, -.05) {$U_1$};
  \node[] at (.2, -.4) {$U_3$};
\end{tikzpicture}
\caption{ }
\label{fig:there is a discontinuous k-loop}
  \end{subfigure}
\caption{ (a) There is a $k$-loop in $U =U_1
    \cup U_2$ connecting $\Gamma_k^{(l)}$ and $\Gamma_k^{(r)}$ although $U = U_1 \cup U_2$ is not
    connected. (b) There is $k$-loop $\gamma$ in $U$. Note that
    $\lineint{u} = 3 (\oint u)_k$ for all fields $u \in L_{\sharp}^2(Y
    ; \C^3)$
  that are curl-free in $U$.}
\end{figure}
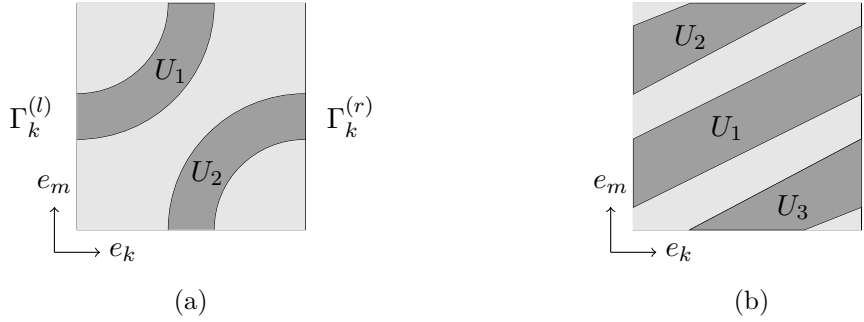

\begin{definition}[Geometric average]\label{definition:geometric average}
 Let $U \subset Y$ be non-empty and such that $\iota(U) \subset \T^3$
 is open. Assume $u \colon U \to \C^3$ is an $L_{\sharp}^2(U;
 \C^3)$-vector field that is curl free in $\iota(U)$. We define its \emph{geometic average} $\oint u \in \C^3$ as follows:
 \begin{enumerate}
 \item[1)] If $U$ is a simple Helmholtz domain, then the vector field
   $u$ can be written as $u = \nabla \Theta + c_0$, where $\Theta \in
   H_{\sharp}^1(U)$ and $c_0 \in \C^3$. In this case, for $k \in
   \set{1,2,3}$, we set 
   \begin{equation*}
     \bigg( \oint u \bigg)_k \coloneqq
     \begin{cases}
       \scalarproduct{c_0}{\e_k} & \textrm{ for } k \in \mathcal{L}_U \, , \\
       0 & \textrm{ otherwise}\, .
     \end{cases}
     \end{equation*}
\item[2)] If for all $k \in \set{1,2,3}$ the path $\gamma_k$ along the
  edge is a $k$-loop in $U$, then, for $k \in \set{1,2,3}$, we set
  \begin{equation*}
    \bigg( \oint u \bigg)_k \coloneqq \int_{\gamma_k} u \, .
  \end{equation*}
 \end{enumerate}
\end{definition}

In later sections, we consider fields $u \colon Y \to \C^3$ that are curl free in $\iota(U)$ and vanish in $Y \setminus \overline{U}$, where $U \subset Y$ is non empty and proper. To define the geometric average of those fields, we restrict $u$ to the subset $U$ and apply the above definition.

\begin{remark}[The geometric average is well defined]
\textit{a)} Let $U$ be a simple Helmholtz domain. Fix $k \in \mathcal{L}_U$, and let $\gamma \colon [0,1]
\to \T^3$ be a $k$-loop in $\iota(U)$. Using that $u = \nabla \Theta + c_0$ in
$U$ as well as the periodicity of $\Theta$, we find that
\begin{equation*}
  \lineint{u} = \lineint{c_0} = \scalarproduct{c_0}{\e_k}
  \scalarproduct{\tilde{\gamma}(1) - \tilde{\gamma}(0)}{\e_k}\, ,
\end{equation*}
and hence the number $\scalarproduct{c_0}{\e_k}$ in \textit{1)} is well defined (despite our observation in Remark~\ref{remark:constant in Helmholtz domain is not unique}).

\textit{b)} The two definitions \textit{1)} and \textit{2)} coincide
when both can be applied. To see this, fix an index $k \in
\set{1,2,3}$. The domain $U$ is a simple Helmholtz domain, and hence
we find a potential $\Theta \in H_{\sharp}^1(U)$ and a constant $c_0
\in \C^3$ such that $u = \nabla \Theta + c_0$ in $U$. We may assume
that $\Theta \in C_{\sharp}^1(U)$; otherwise we approximate by smooth
functions. We then find, for the path $\gamma_k$ along the edge, that
\begin{equation*}
  \int_{\gamma_k}u 
  = \int_{\gamma_k} \nabla \Theta +
  \scalarproduct{c_0}{\e_k} 
  = \Theta(\gamma_k(1)) - \Theta(\gamma_k(0)) + \scalarproduct{c_0}{\e_k}
  = \scalarproduct{c_0}{\e_k}\, .
\end{equation*}
\end{remark}

\begin{remark}
  \textit{a)} Let $U$ be a simple Helmholtz domain, and let $\gamma$
  be a $k$-loop in $U$. We remark that $(\oint u)_k \neq \lineint{u}$
  in general. To give an example, let $\tilde{\gamma}$ be a lift of
  the $k$-loop $\gamma$ that travels the distance $2$ in the $k$th
  direction, that is,
  $\scalarproduct{\tilde{\gamma}(1) - \tilde{\gamma}(0)}{\e_k} = 2$
  (as in Figure \ref{fig:there is a discontinuous k-loop}). In this case,
   $\lineint{u} = 2 (\oint u)_k$. Nevertheless, there always holds
  $\lineint{u} = \lambda_k (\oint u)_k$ for
  $\lambda_k \in \Z \setminus \set{0}$.

  \textit{b)} There are domains for which definition \textit{1)} can be applied, but definition \textit{2)} cannot be used. Indeed, $U \coloneqq B_r(0)$ is a simple Helmholtz domain for $r \in (0, 1/2)$. However, $\gamma_k(t) \notin U$ for all $t \in [0, 1]$ and all $k \in \set{1,2,3}$.

\textit{c)} On the other hand, choosing $Y \setminus U$ to be a $3$-dimensional full torus $\mathbb{S}^1 \times \mathbb{D}^2 \subset \subset Y$, we find that $\gamma_1$, $\gamma_2$, and $\gamma_3$ are $1$-, $2$-, and $3$-loops in $U$, respectively. The domain $U$, however, is not simple Helmholtz.
\end{remark}
\textbf{Properties of the geometric average.} We introduce the function space 
\begin{equation}\label{eq:definition of X}
  X \coloneqq \set{ \fhi \in L_{\sharp}^2(Y ; \C^3) \given \fhi = 0
    \textrm{ in } Y \setminus \overline{U}  \textrm{ and } \div \fhi = 0 \textrm{ in } Y }\, .
\end{equation}
For a simple Helmholtz domain $U$ and under slightly stronger assumptions on the vector field $u \colon Y \to \C^3$, Bouchitt\'e, Bourel, and Felbacq, \cite {BouchitteBourel2009}, define the geometric average $\oint u$ by the identity~\eqref{eq:identity for geometric average} below. 
We show that our definition and theirs agree when both can be applied.

\begin{lemma}\label{lemma:both notions of geometric average coincide}
  Let $U \subset Y$ be a simple Helmholtz domain. If $u \colon Y \to \C^3$ is a vector field of class $L_{\sharp}^2(Y; \C^3)$ that is curl free in $U$, then the identity
  \begin{equation}\label{eq:identity for geometric average}
    \int_Y \scalarproduct{u}{\fhi} = \scalarproduct[\bigg]{\oint u}{\int_Y \fhi} 
  \end{equation}
  holds for all $\fhi \in X$, where the function space $X$ is defined in~\eqref{eq:definition of X}.
\end{lemma}

\begin{proof}
  As $U$ is a simple Helmholtz domain,  by Corollary~\ref{corollary:helmholtz decomposition with vanishing components of c_0}, we find a potential $\Theta \in H_{\sharp}^1(U)$ and a constant $c_0 \in \C^3$ such that $u = \nabla \Theta + c_0$ in $U$ and $\scalarproduct{c_0}{\e_k} = 0$ for all $k \in \mathcal{N}_U$. Substituting this decomposition into the left-hand side of~\eqref{eq:identity for geometric average}, we find that
  \begin{equation*}
    \int_Y \scalarproduct{u}{\fhi} = \int_U \scalarproduct{\nabla \Theta}{\fhi} + \int_U \scalarproduct{c_0}{\fhi} = \int_U \scalarproduct{\nabla \Theta}{\fhi} + \sum_{k \in \mathcal{L}_U} \scalarproduct{c_0}{\e_k}  \scalarproduct[\bigg]{\e_k}{\int_Y \fhi}\, .
  \end{equation*}
 Note that the function space $X$ is a subset of $H_{\sharp}(\div, Y)$. We are thus allowed to use integration by parts in the first integral on the right-hand side. Using that $\fhi$ is divergence free in $Y$ and vanishes in $Y \setminus \overline{U}$, we find that
 \begin{equation*}
   \int_Y \scalarproduct{u}{\fhi} = \sum_{k \in \mathcal{L}_U}
   \scalarproduct{c_0}{\e_k} \scalarproduct[\bigg]{\e_k}{\int_Y \fhi}
   = \scalarproduct[\bigg]{\oint u}{ \int_Y \fhi}\, ,
 \end{equation*}
and hence the claimed equality holds.
\end{proof}

When we derive effective equations, we need the following property of
the geometric average, which is a consequence of the
Lemma~\ref{lemma:both notions of geometric average coincide}.

\begin{corollary}\label{corollary:property of geometric average}
  Assume that $U \subset Y$ is a simple Helmholtz domain, and let $u \colon Y \to \C^3$ be a vector field of class $L_{\sharp}^2(Y ; \C^3)$ that is curl free in $U$. If $E \colon Y \to \C^3$ is another field of class $L_{\sharp}^2(Y; \C^3)$ that is curl free and that vanishes in $Y \setminus \overline{U}$, then 
  \begin{equation}\label{eq:property of the geometric average}
    \int_{U} (u \wedge E) = \bigg( \oint u \bigg) \wedge \bigg( \int_{Y} E\bigg) \, .
  \end{equation}
\end{corollary}

\begin{proof}
 Fix $k \in \set{1,2,3}$. Defining the field $ \fhi \colon Y \to \C^3$ by $\fhi \coloneqq E \wedge \e_k$ provides us with an element of $X$; indeed, $\fhi$ is of class $L_{\sharp}^2(Y ; \C^3)$ and vanishes in $Y \setminus \overline{U}$. Moreover, for every $\phi \in C_c^{\infty}(Y)$, we calculate that
\begin{equation*}
  \int_Y \scalarproduct{\fhi}{\nabla \phi} 
  = - \int_Y \scalarproduct{E}{\nabla \phi \wedge \e_k} 
  =  - \int_Y \scalarproduct{E}{ \curl (\phi \, \e_k)} = 0.
\end{equation*}
That is, $\fhi$ is divergence free in $Y$. We can thus apply Lemma~\ref{lemma:both notions of geometric average coincide} and obtain that
\begin{equation*}
  \int_{U} \scalarproduct{u \wedge E}{\e_k} 
  =  \int_{Y} \scalarproduct{u}{\fhi} 
  =  \scalarproduct[\bigg]{\oint u}{ \int_Y( E \wedge \e_k)}
  =  \scalarproduct[\bigg]{\oint u}{ \bigg(\int_Y E\bigg) \wedge \e_k}\, .
\end{equation*}
As $k \in \set{1,2,3}$ was chosen arbitrarily, this yields~\eqref{eq:property of the geometric average}.
\end{proof}

\begin{remark}
  The statement of the corollary remains true when we replace the assumption on $U$ to be a simple Helmholtz domain by the assumption that $\overline{U} \cap \del Y = \emptyset$. The geometric average $\oint u$ is then given by the second part of Definition~\ref{definition:geometric average}. A proof of the corollary in this situation is given in~\cite {Lipton-Schweizer}.
\end{remark}

\section{Cell problems and their analysis}\label{sec:cell problems}

In this section, we study sequences $(\Eeta)_{\eta}$ and
$(\Heta)_{\eta}$ of solutions to~\eqref{eq:MaxSysSeq-intro} and their two-scale limits $E_0$
and $H_0$. 
\subsection{Cell problem for $E_0$}
\begin{lemma}[Cell problem for $E_0$]\label{lemma:cell problem E_0}
  Let  $R \subset \subset \Omega \subset \R^3$ and $\Sigma \subset
  Y$ be as in Section~\ref{section:geometry and assumptions}, and let
  $(\Eeta, \Heta)_{\eta}$ be a sequence of solutions
  to~\eqref{eq:MaxSysSeq-intro} that satisfies the
  energy-bound~\eqref{eq:energy-bound}. The
  two-scale limit  $E_0 \in L^2(\Omega \times Y; \C^3)$ satisfies the following:
\begin{enumerate}
  \item[i)]
 For almost all $x \in R$ the field $E_0 = E_0(x, \cdot)$ is an element of $ H_{\sharp}(\curl, Y)$ and a distributional solution to
\begin{subequations}\label{eq:cell problem E}
  \begin{align}
    \curly E_0 &= 0 \quad \textrm{ in } Y\, , \label{eq:cell problem E 1}\\
    \divy E_0  &= 0 \quad \textrm{ in } Y\setminus \overline{\Sigma}\, , \label{eq:cell problem E 2}\\
           E_0 &= 0 \quad \textrm{ in } \Sigma\, . \label{eq:cell problem E 3}
  \end{align}
\end{subequations}
Outside the meta-material, the two-scale limit $E_0$ is
$y$-independent; that is,  $E_0(x, y) = E_0(x) = E(x)$ for a.e. $x
\in \Omega \setminus R$ and a.e. $y \in Y$.
\item[ii)]
Given $c \in \C^3$, there exists at most one solution $u
\in L_{\sharp}^2(Y; \C^3)$ to~\eqref{eq:cell problem E} with
cell-average $\meanint_{Y}u(y) \d y = c$.
\end{enumerate}
\end{lemma}

\begin{proof}
\textit{i)} The derivation of the cell problem is by now standard and can, for instance, be found in~\cite{Lamacz-Schweizer-2016}. To give some ideas, fix $x \in \Omega$ and $\eta > 0$, and set $\fhi_{\eta}(x) \coloneqq \fhi(x, x/ \eta)$, where $\fhi \in C_c^{\infty}(\Omega; C_{\sharp}^{\infty}(Y; \C^3))$. Using integration by parts and the two-scale convergence of $(\Eeta)_{\eta}$ we obtain  that 
\begin{equation*}
  \lim_{\eta \to 0} \int_{\Omega} \scalarproduct[\big]{\eta \curl \Eeta(x)}{\fhi_{\eta}(x)} \d x = \int_{\Omega} \int_{Y} \scalarproduct[\big]{E_0(x, y)}{\curly \fhi(x, y)} \d y \d x \, .
\end{equation*}
From this and~\eqref{eq:MaxSeq1-intro}, we deduce that $E_0$ is a distributional solution to~\eqref{eq:cell problem E 1}.
Equation~\eqref{eq:cell problem E 2} is a consequence of~\eqref{eq:MaxSeq2-intro}, and~\eqref{eq:cell problem E 3} follows from~\eqref{eq:MaxSeq3-intro}.
Due to~\eqref{eq:cell problem E 1}, there holds $E_0 \in H_{\sharp}(\curl, Y)$.

In $\Omega \setminus \overline{\Sigma}$, the cell problem for $E_0 = E_0(x, \cdot)$ coincides with~\eqref{eq:cell problem E} if we replace $\Sigma$ by the empty set $\emptyset$. This problem, however, has only the constant solution.

\textit{ii)}
Let $u \in L_{\sharp}^2(Y; \C^3)$ be a distributional solution to~\eqref{eq:cell problem E} with $\meanint_{Y} u = 0$.
We claim that the field $u$ vanishes identically in $Y$.
Indeed,~\eqref{eq:cell problem E 1} implies the existence of a potential $\Theta \in H_{\sharp}^1(Y)$ and of a constant $c_0 \in \C^3$ such that $u = \nabla \Theta + c_0$ in $Y$. We may assume that $\meanint_Y \Theta = 0$. As $u$ has vanishing average, we conclude that $c_0 = 0$. On account of~\eqref{eq:cell problem E 2} and~\eqref{eq:cell problem E 3}, the potential $\Theta$ is a distributional solution to
\begin{align*}
  - \Delta \Theta &= 0 \quad \textrm{ in } Y \setminus \overline{\Sigma}\, ,\\
           \Theta &= d \quad \textrm{ in } \Sigma\, ,
\end{align*}
for some constant $d \in \C$. As $ \Theta \in H_{\sharp}^1(Y)$, the potential $\Theta$ does not jump across the boundary $\del \Sigma$. Consequently,  $\Theta = d$ in $Y$ and thus $u$ vanishes in $Y$.
\end{proof}
While we obtained the uniqueness result immediately, the existence
statement is more involved.  To investigate the solution space to the
cell problem~\eqref{eq:cell problem E}, we use the two index sets
from~\eqref{eq:index sets for H}:
$\mathcal{L}_{\Sigma} = \set[\big]{k \in \set{1,2,3} \given
  \textrm{there is a $k$-loop in } \Sigma}$ and
$\mathcal{N}_{\Sigma} = \set[\big]{k \in \set{1,2,3} \given
  \textrm{there is no $k$-loop in } \Sigma} = \set{1,2,3}
\setminus \mathcal{L}_{\Sigma}$ .  We claim that $\abs{\NE}$ coincides
with the dimension of the solution space of~\eqref{eq:cell problem E}.

\begin{lemma}[Connection between the $E_0$-problem~\eqref{eq:cell
    problem E} and $\NE$]\label{lemma:characterisation of index set N
    E}
  For $k \in \NE$ there exists a unique solution $E^k$
  to~\eqref{eq:cell problem E} with volume average $\e_k$. On the
  other hand, if $c \in \C^3$ is such that $c \notin \spann \set{ \e_k
    \given k \in \NE}$, then there is no solution $E$ to
  \eqref{eq:cell problem E} with volume average $c$.
\end{lemma}

\begin{proof}
\textit{Part 1.} Let $k$ be an element of $\NE$. Our aim is to show
that a solution $E^k$ exist. Due to Lemma~\ref{lemma:existence of
  potentials in case of no k-loop}, there exists a potential
$\tilde{\Theta}_k \in H_{\sharp}^1(\Sigma)$ such that $\nabla
\tilde{\Theta}_k = - \e_k$. 

We extend $\tilde{\Theta}_k$ to all of $Y$ as follows: Let $\Theta_k \in H_{\sharp}^1(Y)$ be the weak solution to
  \begin{align*}
    - \Delta \Theta_k &= 0 \quad \: \: \: \, \textrm{ in } Y \setminus \overline{\Sigma} \, , \\
    \Theta_k &=\tilde{\Theta}_k  \quad \textrm{ on } \overline{\Sigma}\, .
  \end{align*}
By setting $E^k \coloneqq \nabla \Theta_k + \e_k$, we obtain a
solution to~\eqref{eq:cell problem E} whose cell average is $\e_k$.

  \textit{Part 2.}
 Let $c \in \C^3$ be such that $c \notin \spann \set{ \e_k
    \given k \in \NE}$. Assume that there is a solution $E$ to
  \eqref{eq:cell problem E} with $\meanint_Y E = c$ . 
  By Part 1, for every $k \in \NE$
  there is a solution $E^k$ to~\eqref{eq:cell problem E} with
  $\meanint_Y E^k = \e_k$. 
  Consider the field 
  \begin{equation*}
    v \coloneqq E - \sum_{k \in \NE} \scalarproduct{c}{\e_k}E^k \, .
  \end{equation*}
  This field is a solution to \eqref{eq:cell problem E} with
  $\meanint_Y v = \sum_{k \in \LE} \scalarproduct{c}{\e_k} \e_k$. As
  $c \notin \spann \set{ \e_k \given k \in \NH}$, there holds
  $\meanint_Y v \neq 0$. We find a potential
  $\Phi \in H_{\sharp}^1(Y)$ such that
  $v = \nabla \Phi + \sum_{l \in \LE} \scalarproduct{c}{\e_l} \e_l$ in
  $Y$ since $\curl v = 0$ in $Y$.  Fix an index $k \in \LE$, and let $\gamma$ be a $k$-loop in
  $\Sigma$. By~\eqref{eq:cell problem E 3}, we have that
  $\nabla \Phi \in C_{\sharp}^0(\Sigma; \C^3)$. As $v = 0$ in
  $\Sigma$, we calculate, by exploiting the periodicity of $\Phi$ in
  $Y$,
\begin{equation}\label{eq:proof characterisation of solution space for E - 2}
  0 = \lineint{v} = \lineint{\nabla \Theta} + \sum_{l \in \LE} \scalarproduct{c}{\e_l}
  \lineint{\e_l} = \scalarproduct{c}{\e_k} \scalarproduct{\tilde{\gamma}(1) -
    \tilde{\gamma}(0)}{\e_k} \, .
\end{equation}
Note that $ \scalarproduct{\tilde{\gamma}(1) -
    \tilde{\gamma}(0)}{\e_k} \neq 0$ since $\gamma$ is a $k$-loop. Thus $\scalarproduct{c}{\e_k} = 0$ for all $k \in \LE$.
This, however, contradicts $\meanint_Y v \neq 0$.
\end{proof}
Thanks to this lemma, we have the following result.

\begin{proposition}[Characterisation of the solution space of the $E_0$-problem]\label{proposition: dimension of solution space to the cell problem of E}
  For every index $k \in \NE$, let $E^k = E^k(y)$ be the solution
  to~\eqref{eq:cell problem E} from Lemma~\ref{lemma:characterisation
    of index set N E}.  Then there holds: Every solution $u$ to~\eqref{eq:cell
    problem E} can be written as 
  \begin{equation}\label{eq:characterisation of the solution space to E problem}
     u = \sum_{k \in \NE} \alpha_k E^k\,
  \end{equation}
  with constants $\alpha_k \in \C$ for $k \in \NE$.
 In particular, the dimension of the solution space coincides with $\abs{\NE}$.
\end{proposition}

\begin{proof}
 We need to prove that
  every solution $u$ to~\eqref{eq:cell problem E} can be written as in~\eqref{eq:characterisation of the solution space to E problem}.

  Let $u \in H_{\sharp}(\curl, Y)$ be an arbitrary solution
  to~\eqref{eq:cell problem E}. On account of~\eqref{eq:cell problem E
    1}, we find a potential $\Theta \in H_{\sharp}^1(Y)$ and $c_0 \in
  \C^3$ such that  $u = \nabla \Theta + c_0$ in $Y$. For each $k \in
  \set{1,2,3}$, we set $\alpha_k \coloneqq \scalarproduct{c_0}{\e_k}$.
  Consider the field
  \begin{align}
    v \coloneqq u - \sum_{k \in \NE} c_k E^k\, . \nonumber
  \shortintertext{This field is a solution to~\eqref{eq:cell problem E} with}
  \label{eq:proof characterisation of solution space for E - 1}
    \meanint_Y v =\sum_{l \in \LE} \alpha_l \e_l \, .
  \end{align}
  By the second statement of Lemma~\ref{lemma:characterisation of
    index set N E}, the coefficient $\alpha_l = 0$ for all $l \in \LE$. Hence  $v = 0$
in $Y$ by the uniqueness result from Lemma \ref{lemma:cell problem E_0}. This provides
\eqref{eq:characterisation of the solution space to E problem}.
\end{proof}

\begin{remark}\label{remark:the solution space of the cell problem of E is at most three dimensional}
  The previous proposition implies, in particular, that the solution space to the cell
  problem of $E_0$ is at most three dimensional. Note that no
  assumption (such as simple connectedness) on the domain $\Sigma$ was imposed here.
\end{remark}

\subsection{Cell problem for $H_0$}

\begin{lemma}[Cell problem for $H_0$]            
  \label{lemma:cell problem for H_0}
   Let  $R \subset \subset \Omega \subset \R^3$ and $\Sigma \subset
  Y$ be as in Section~\ref{section:geometry and assumptions}, and let
  $(\Eeta, \Heta)_{\eta}$ be a sequence of solutions
  to~\eqref{eq:MaxSysSeq-intro} that satisfies the
  energy-bound~\eqref{eq:energy-bound}. The
  two-scale limit  $H_0 \in L^2(\Omega \times Y; \C^3)$ satisfies the following:
\begin{enumerate}
\item[i)]
For almost all $x \in R$ the field $H_0 = H_0(x , \cdot)$ is an element of $H_{\sharp}(\div, Y)$ and a distributional solution to
   \begin{subequations}\label{eq:cell problem H_0}
     \begin{align}
       \label{eq:cell problem H_0 1}
       \curly H_0   &= 0 \quad \textrm{ in } Y \setminus \overline{\Sigma}\, , \\
       \label{eq:cell problem H_0 2}
                               \divy H_0 &= 0  \quad \textrm{ in } Y\, ,\\
       \label{eq:cell problem H_0 3}
       H_0 &= 0 \quad \textrm{ in } \Sigma\, .
  \end{align}
\end{subequations}
Outside the meta-material, the two-scale limit $H_0$ is
$y$-independent; that is,  $H_0(x, y) = H_0(x) = H(x)$ for a.e. $x
\in \Omega \setminus R$ and a.e. $y \in Y$.
\item[ii)]
If $Y \setminus \overline{\Sigma}$ is a simple Helmholtz domain,
then for every $c \in \C^3$ there is at most one solution $u \in
H_{\sharp}(\div, Y)$ to~\eqref{eq:cell problem H_0} with geometric
average $\oint u = c$.
\end{enumerate}
\end{lemma}

\begin{proof}
\textit{i) $H_0$ is a distributional solution to~\eqref{eq:cell problem H_0}.}
Exploiting  the two-scale convergence of $(\Heta)_{\eta}$ and
$(\Eeta)_{\eta}$, we deduce~\eqref{eq:cell problem H_0 1} by
Maxwell's
equation~\eqref{eq:MaxSeq2-intro}. By~\eqref{eq:MaxSeq1-intro}, each
$\Heta$ is a divergence-free field in $\Omega$, and hence $\divy H_0 = 0$ in $Y$.
On account of~\eqref{eq:MaxSeq3-intro},  the field $\Heta\,
\indicator{\Sigma_{\eta}} $ vanishes identically in $R$. Thus,  $0 =
\Heta \, \indicator{\Sigma_{\eta}} \twoscale H_0 \,
\indicator{\Sigma}$ implies that $H_0(x, y) = 0$ for almost all $(x,
y) \in R \times \Sigma$, and hence~\eqref{eq:cell problem H_0 3}.

Outside the meta material, the fields $H_0 = H_0(x, \cdot)$ and $H$ coincide since the corresponding cell problem admits only constant solutions.

\textit{ii) Uniqueness.} Let $Y \setminus \overline{\Sigma}$ be a
simple Helmholtz domain, and let $u \in L_{\sharp}^2(Y; \C^3)$ be a
solution to~\eqref{eq:cell problem H_0} with vanishing geometric
average, $\oint u = 0$. We claim that $u$ vanishes identically in
$Y$. As $u$ is curl free in the simple Helmholtz domain $Y \setminus
\overline{\Sigma}$, we  find a potential $\Theta  \in H_{\sharp}^1(Y
\setminus \overline{\Sigma})$ and a constant $c_0 \in \C^3$ such
that $u = \nabla \Theta + c_0$ in $Y \setminus \overline{\Sigma}$. For
an index $k \in \LH$, we can apply the first part of
Definition~\ref{definition:geometric average} and find that
$(\oint u)_k =\scalarproduct{c_0}{\e_k} $. By assumption, $\oint u =
0$ and hence $\scalarproduct{c_0}{\e_k} = 0$ for all $k \in \LH$. Due
to Lemma~\ref{lemma:existence of potentials in case of no k-loop}, for
every $k \in \NH$, we find a potential $\Theta_k \in H_{\sharp}^1(Y
\setminus \overline{\Sigma})$
such that $\nabla \Theta_k = \e_k$. The function $\tilde{\Theta} \coloneqq
\Theta + \sum_{k \in \NH} \scalarproduct{c_0}{\e_k} \Theta_k$ is an
element of $H_{\sharp}^1(Y \setminus \overline{\Sigma})$. Moreover, $u
= \nabla \Theta + \sum_{k \in \NH} \scalarproduct{c_0}{\e_k} =
\nabla \tilde{\Theta}$ in $Y \setminus \overline{\Sigma}$.
Equations~\eqref{eq:cell problem H_0 2} and~\eqref{eq:cell problem H_0
  3} imply $0 = \scalarproduct{u}{\nu} = \del_{\nu} \tilde{\Theta}$ on
$\del \Sigma$, where $\nu$ is the outward unit normal vector. We conclude that $\tilde{\Theta}$ is a weak solution to 
\begin{align*}
  - \Delta \tilde{\Theta} &= 0 \quad \textrm{ in } Y \setminus \overline{\Sigma}\, ,\\
  \del_{\nu} \tilde{\Theta} &= 0 \quad \textrm{ on } \del \Sigma\,.
\end{align*}
Solutions to this Neumann boundary problem are constant since $Y
\setminus \overline{\Sigma}$ is a domain. Hence $u = 0$ in $Y$.
\end{proof}

\begin{remark}
  Note that, in contrast to the $E_0$-problem (see
  Lemma~\ref{lemma:cell problem E_0}), the uniqueness statement of
  \textit{ii)} is false if we do not assume that
  $Y \setminus \overline{\Sigma}$ is a simple Helmholtz
  domain. Indeed, in~\cite {Lipton-Schweizer}, a $3$-dimensional full
  torus $\Sigma$ is studied and a non-trivial solution with vanishing
  geometric average is found.
\end{remark}

\begin{lemma}[Connection between the $H_0$-problem \eqref{eq:cell problem H_0} and $\LH$]\label{lemma:for each k in L H there is a solution with geometric average equal to e k}
  Let $Y \setminus \overline{\Sigma}$ be a simple Helmholtz domain. If
  the index $k \in \set{1,2,3}$ is an element of $\LH$, then there
  exists a unique solution $H^k$ to~\eqref{eq:cell problem H_0} with
  geometric average $\e_k$.
\end{lemma}

\begin{proof}
  Fix $k \in \LH$ and let $\Theta_k \in H_{\sharp}^1(Y \setminus \overline{\Sigma})$ be a distributional solution to
  \begin{align*}
    - \Delta \Theta_k &= 0 \qquad \qquad \textrm{ in } Y \setminus \overline{\Sigma}\, , \\
    \del_{\nu} \Theta_k &= - \scalarproduct{\e_k}{\nu} \quad \textrm{ on } \del \Sigma\, .
  \end{align*}
We define $H^k \colon Y \to \C^3$ by
\begin{equation*}
  H^k \coloneqq 
  \begin{cases}
    \nabla \Theta_k + \e_k & \textrm{ in } Y \setminus \overline{\Sigma} \, ,\\
    0                   & \textrm{ in } \Sigma\, .
   \end{cases}
\end{equation*}
 In this way, we obtain an $L_{\sharp}^2(Y ; \C^3)$-vector field $H^k$
 that is a  distributional solution to~\eqref{eq:cell problem
   H_0}. As $k \in \LH$, we obtain, using the definition of the geometric average, $\oint  H^k = \e_k$. 
\end{proof}

\begin{proposition}[Characterisation of the solution space to the $H_0$-problem]\label{proposition:for every element of the set L E there is a solution to the cell problem}
 Let $Y \setminus \overline{\Sigma}$ be a simple Helmholtz domain. For
 every index $k \in \LH$, let $H^k = H^k(y)$ be the solution
 to~\eqref{eq:cell problem H_0} from Lemma~\ref{lemma:for each k in L
   H there is a solution with geometric average equal to e k}. Then
 there holds: Every solution $u$ to~\eqref{eq:cell problem H_0} can be
 written as
  \begin{equation}\label{eq:solution to H-problem can be written as
      linear combination of special fields}
    u = \sum_{k \in \LH} \alpha_k H^k\, ,
  \end{equation}
with  constants $\alpha_k \coloneqq (\oint u)_k\in \C$ for $k \in \LH$.
In particular, the dimension of the solution space coincides with $\abs{\LH}$.
\end{proposition}

\begin{proof}
We use the solutions $H^k$ of Lemma~\ref{lemma:for each k in L H there is a solution with
  geometric average equal to e k}. The set $\set{H^k
  \given k \in \LH}$ is linearly independent since the geometric
averages of the $H^k$ are linear independent. We need to prove that every
solution $u$ to~\eqref{eq:cell problem H_0} can be written as in~\eqref{eq:solution to H-problem can be written as
      linear combination of special fields}.

Let $u \in H_{\sharp}(\div, Y)$ be a solution to~\eqref{eq:cell
  problem H_0}. We define 
\begin{align*}
  v \coloneqq u - \sum_{k \in \LH} \bigg(\oint u\bigg)_k H^k\, .
\shortintertext{The field $v$ is also a solution to~\eqref{eq:cell
  problem H_0} and has the geometric average}
  \oint v = \sum_{k \in \NH} \bigg( \oint u \bigg)_k \e_k \, .
\end{align*}
As those components of the geometric average for which there is no
loop in $Y \setminus \overline{\Sigma}$ vanish, by the first part of
the definition of the geometric average, the right-hand side
vanishes. Due to the uniqueness statement of Lemma~\ref{lemma:cell
  problem for H_0}, $v$ vanishes in $Y$. This provides~\eqref{eq:solution to H-problem can be written as
      linear combination of special fields}.
\end{proof}

\begin{remark}
  As a consequence of the previous proposition, we find that the solution space to the cell problem of $H_0$ is at most three dimensional if $Y \setminus\overline{\Sigma}$ is a simple Helmholtz domain.
\end{remark}

\begin{remark}
  Geometric intuition suggests that $\abs{\LH} \geq \abs{\NE}$; there
  is, however, no obvious proof of this fact. As a consequence of this
  inequality, we find that the dimensions $d_E$ and $d_H$ of the
  solution spaces to the $E_0$-problem~\eqref{eq:cell problem E} and
  to the $H_0$-problem~\eqref{eq:cell problem H_0} satisfy $d_H \geq d_E$.
\end{remark}

\section{Derivation of the effective equations}\label{sec:derivation of the effective equations}

Our aim in this section is to derive the effective Maxwell system. We
assume that $\Omega \subset \R^3$, the subdomain $R \subset \subset
\Omega$, and $\Sigma \subset Y$ are as in
Section~\ref{section:geometry and assumptions}. We work with the two
index sets $\NE$ and $\LH$ defined in~\eqref{eq:index sets for H 2}
and \eqref{eq:index sets for H 1}, respectively.

We define the matrices $\eps_{\textrm{eff}}$, $\mu_{\textrm{eff}} \in
\R^{3 \times 3}$ by setting, for $k \in \set{1,2,3}$, 
\begin{align}
  (\eps_{\textrm{eff}})_{kl} &\coloneqq \label{eq:definition eps eff}
  \begin{cases}
  \scalarproduct[\big]{E^k}{E^l}_{L^2(Y; \C^3)} & \textrm{ if  } k, l \in \NE\, ,\\
  0                                         & \:\, \textrm{otherwise} \, ,
  \end{cases}
\shortintertext{and}
  \mu_{\textrm{eff}} \, \e_k &\coloneqq \label{eq:definition mu eff}
  \begin{cases}
    \int_Y H^k \quad \textrm{ if } k \in \LH\, ,\\
    0 \qquad \quad \textrm{ otherwise}\, .
  \end{cases}
\end{align}
The  \emph{effective permittivity} $\hat{\eps} \colon \Omega \to \R^{3\times 3}$  and the \emph{effective permeability} $\hat{\mu} \colon \Omega \to \R^{3 \times 3}$ are defined by
\begin{equation}\label{eq:definition effective permittivity and permeability}
  \hat{\eps}(x) \coloneqq
  \begin{cases}
    \eps_{\textrm{eff}} \quad \textrm{if } x \in R\, ,\\
    \id_{3} \:\: \textrm{ otherwise}\, ,
  \end{cases}
  \qquad 
  \hat{\mu}(x) \coloneqq
  \begin{cases}
     \mu_{\textrm{eff}} \quad \textrm{if } x \in R\, ,\\
    \id_{3} \:\: \,\textrm{ otherwise}\, ,
  \end{cases}
\end{equation}
where $\id_3 \in \R^{3 \times 3}$ is the identity matrix. 

Let $(\Eeta, \Heta)_{\eta}$ be a sequence of solutions
to~\eqref{eq:MaxSysSeq-intro} that satisfies the
energy-bound~\eqref{eq:energy-bound}; the corresponding two-scale
limits are denoted by $E_0$, $H_0 \in L^2(\Omega \times Y;
\C^3)$. Assume that $Y \setminus \overline{\Sigma}$ is such that we
can define the geometric average.  We then define the \emph{limit
  fields} $\hat{E},\hat{H} \colon \Omega \to \C^3$ by
\begin{equation}\label{eq:limit fields}
  \hat{E}(x) \coloneqq \int_Y E_0(x, y) \d y \quad \textrm{ and } \quad \hat{H}(x) \coloneqq \oint H_0(x, \cdot)\, .
\end{equation}
We recall that $H_0(x, \cdot)$ solves~\eqref{eq:cell problem H_0}; it is
therefore curl free in $Y \setminus \overline{\Sigma}$ and vanishes in $\Sigma$.
The two fields $\hat{E}$ and $\hat{H}$ are of class $L^2(\Omega; \C^3)$.

\begin{theorem}[Macroscopic equations]\label{theorem:macroscopic equations}
  Let $(\Eeta, \Heta)_{\eta}$ be a sequence of solutions
  to~\eqref{eq:MaxSysSeq-intro} that satisfies the energy
  bound~\eqref{eq:energy-bound}. Assume that the geometry
  $\Sigma_{\eta} \subset R \subset \Omega$ is as in
  Section~\ref{section:geometry and assumptions}.  Let
  $Y \setminus \overline{\Sigma}$ be a simple Helmholtz domain, and
  let $\NE$ and $\LH$ be the corresponding index sets. Let the
  effective permittivity $\hat{\eps}$ and the effective permeability
  $\hat{\mu}$ be defined as in~\eqref{eq:definition effective
    permittivity and permeability}, and let the limit fields
  $(\hat{E}, \hat{H})$ be defined as in~\eqref{eq:limit fields}. Then
  $\hat{E}$ and $\hat{H}$ are distributional solutions to
\begin{subequations}\label{eq:effective equations - derivation of effective equations}
  \begin{align}
        \label{eq:effective equations 1 - derivation of effective equations}
        \curl \hat{E}   &=\phantom{-} \i \omega \mu_0 \hat{\mu} \hat{H}      \quad \:\:\: \:\textrm{ in } \Omega\, ,\\
         \label{eq:effective equations 3 - derivation of effective equations}
          \curl \hat{H} &= - \i \omega \eps_0 \hat{\eps} \hat{E} \quad
                          \: \: \: \: \: \,  \textrm{ in } \Omega \setminus R\, , \\
        \label{eq:effective equations 2 - derivation of effective equations}
    (\curl \hat{H})_k &= -\i \omega \eps_0 (\hat{\eps} \hat{E})_k
                        \quad \textrm{ in } \Omega\, ,  \textrm{ for
                        every } k \in \NE \, , \\
        \label{eq:effective equations 4 - derivation of effective equations}
        \hat{E}_k          &= 0 \qquad \qquad \quad \: \: \: \, \textrm{ in } R\, , \textrm{ for every }  k \in \set{1,2,3} \setminus \NE\, , \\
        \label{eq:effective equations 5 - derivation of effective equations}
        \hat{H}_k         &= 0 \qquad \qquad \quad \: \: \: \, \textrm{ in } R \, , \textrm{ for every } k \in \NH \, .
  \end{align}
\end{subequations}
\end{theorem}
\begin{proof}
 Thanks to the preparations
  of the last section, we can essentially
  follow~\cite{Lipton-Schweizer} to derive~\eqref{eq:effective equations 1 - derivation of effective equations}--\eqref{eq:effective equations 2 - derivation of effective equations}. The remaining
  relations \eqref{eq:effective equations 4 - derivation of effective
    equations} and \eqref{eq:effective equations 5 - derivation of
    effective equations} follow from the characterisation of the
  solution spaces of the cell problems.

 \textit{Step 1: Derivation of~\eqref{eq:effective equations 1 -
      derivation of effective equations}.} 
  The distributional limit of~\eqref{eq:MaxSeq1-intro} reads
  \begin{equation}\label{eq:distributional limit - proof of effective equations}
    \curl E = \i \omega \mu_0 H \quad \textrm{ in } \Omega\, .
  \end{equation}
  We recall that $E$ and $H$ are the weak
  $L_{\sharp}^2(Y;\C^3)$-limits of $(\Eeta)_{\eta}$ and
  $(\Heta)_{\eta}$, respectively.  By the definition of the limit
  $\hat{E}$ in \eqref{eq:limit fields} and the volume-average property of the two-scale limit
  $E_0$, we find that
\begin{equation*}
  \hat{E}(x) = \int_Y E_0(x, y) \d y = E(x)\, ,
\end{equation*}
for a.e. $x \in \Omega$.
Thus, $\curl E = \curl \hat{E}$.

On the other hand, using that $Y \setminus \overline{\Sigma}$ is a
simple Helmholtz domain, for almost all $(x,y) \in R\times Y$, the
two-scale limit $H_0$ can be written with coefficients $H_k(x)$ as
\begin{equation}\label{eq:proof of effective equations - 2}
  H_0(x,y) = \sum_{k \in \LH} H_k(x) H^k(y) \, ,
\end{equation}
by Proposition~\ref{proposition:for every element of the set L E there
  is a solution to the cell problem}. The averaging property of the
two-scale limit, the identity~\eqref{eq:proof of effective equations - 2}, and the definition of $\mu_{\eff}$ imply that 
\begin{equation}\label{eq:proof effective equations - 1}
  H(x) = \int_Y H_0(x, y) \d y 
       = \sum_{k \in \LH}H_k(x) \int_Y H^k(y) \d y 
       = \mu_{\textrm{eff}} \sum_{k \in  \LH} H_k(x) \e_k\, .
\end{equation}
Using the definition of the limit field $\hat{H}$ in~\eqref{eq:limit fields} and
identity~\eqref{eq:proof of effective equations - 2}, we conclude
from~\eqref{eq:proof effective equations - 1} that 
\begin{equation}\label{eq:proof of effective equations - 6}
  H(x) = \mu_{\textrm{eff}} \hat{H}(x) \, .
\end{equation}
Outside the meta-material the $L^2(Y;\C^3)$-weak limit $H$ and the
two-scale limit $H_0$ coincide due to Lemma~\ref{lemma:cell problem
  for H_0}. Moreover, $\hat{\mu}$ equals the identity, and
hence~\eqref{eq:proof of effective equations - 6} holds also in $\Omega \setminus R$. From~\eqref{eq:proof of effective equations - 6}
  and~\eqref{eq:distributional limit - proof of effective equations}
  we conclude \eqref{eq:effective equations 1 - derivation of effective equations}.

\textit{Step 2: Derivation of~\eqref{eq:effective equations 3 - derivation of effective equations}, ~\eqref{eq:effective equations 4 - derivation of effective equations} and~\eqref{eq:effective equations 5 - derivation of effective equations}.}
To prove~\eqref{eq:effective equations 3 - derivation of effective equations}, we first observe that $\Omega \setminus R \subset \Omega \setminus \Sigma_{\eta} $. We can therefore take the distributional limit in~\eqref{eq:MaxSeq2-intro} as $\eta$ tends to zero and obtain that
\begin{equation*}
  \curl H = - \i \omega \eps_0 E \quad \textrm{ in } \Omega \setminus R \, .
\end{equation*}
This shows~\eqref{eq:effective equations 3 - derivation of effective equations}, since $\hat{H} = H$ and $\hat{E} = E$ in $\Omega \setminus R$.

By Proposition~\ref{proposition: dimension of solution space to the
  cell problem of E}, the two-scale limit $E_0$ can be written with
coefficients $E_k(x)$ as $E_0(x, y) = \sum_{k \in \NE} E_k(x)
E^k(y)$. Due to the definition of the limit field $\hat{E}$
in~\eqref{eq:limit fields}, we find that
\begin{equation}\label{eq:proof effective equations - 7}
  \hat{E}(x) =  \int_{Y} \sum_{k \in \NE} E_k(x) E^k(y) \d y = \sum_{k \in \NE} E_k(x) \e_k\, , 
\end{equation}
for $x \in R$. Consequently, equation~\eqref{eq:effective equations 4 - derivation of effective equations} holds. Similarly, the definition of $\hat{H}$ and Proposition~\ref{proposition:for every element of the set L E there is a solution to the cell problem} imply that
\begin{equation*}
  \hat{H}(x) = \oint \sum_{k \in \LH} H_k(x) H^k(y) \d y = \sum_{k \in \LH} H_k(x) \e_k
\end{equation*}
for almost all $x \in R$. This proves~\eqref{eq:effective equations 5 - derivation of effective equations}.

\textit{Step 3: Derivation of~\eqref{eq:effective equations 2 -
    derivation of effective equations}.}  We use the defining property
of the two-scale convergence and appropriate oscillating test
functions. For $k \in \NE$ and $\theta \in C_c^{\infty}(\Omega; \R)$,
we set $\fhi(x, y) \coloneqq \theta(x) E^k(y)$ for $x\in \Omega$ and
$y \in Y$. We use $\fhi_{\eta}(x) \coloneqq \fhi(x, x/ \eta)$ for
$x \in \Omega$. From the two-scale convergence of $(\Heta)_{\eta}$, we
obtain that
\begin{align*}
  \lim_{\eta \to 0}\int_{\Omega}\scalarproduct{\Heta}{\curl \fhi_{\eta}}
  &= \int_{\Omega} \int_Y \scalarproduct[\big]{H_0(x,y)}{\nabla \theta(x) \wedge E^k(y)}\d y\d x\\
  &=- \int_{\Omega} \scalarproduct[\Big]{\nabla \theta (x)}{\int_Y H_0(x,y) \wedge E^k(y) \d y} \d x \, .
  \end{align*}
Thanks to Corollary~\ref{corollary:property of geometric average} on the geometric average, the identity
 \begin{equation*}
   \int_Y H_0(x,y) \wedge E^k(y) \d y = \oint H_0(x, \cdot) \wedge \e_k 
 \end{equation*}
holds for almost all $x \in \Omega$. Consequently, 
  \begin{align}
    \lim_{\eta \to 0}\int_{\Omega}\scalarproduct{\Heta}{\curl \fhi_{\eta}}
  &= -\int_{\Omega} \scalarproduct[\Big]{\nabla \theta(x)}{\oint H_0(x, \cdot) \wedge \e_k} \d x \nonumber\\
  &= \int_{\Omega} \scalarproduct[\big]{\hat{H}(x)}{\curl \big(\theta(x) \e_k \big)} \d x\, . \label{eq:proof of effective equations - 3}
\end{align}
On the other hand, $\Heta$ is a distributional solution to Maxwell's
equation~\eqref{eq:MaxSeq2-intro}. Thus, using~\eqref{eq:proof effective equations - 7}
\begin{align}
  \lim_{\eta \to 0} \int_{\Omega} \scalarproduct{\Heta}{\curl \fhi_{\eta}} 
  &= - \i \omega \eps_0 \lim_{\eta \to 0} \int_{\Omega} \scalarproduct{\Eeta}{\fhi_{\eta}} \nonumber\\
  &= - \i \omega \eps_0 \int_{\Omega} \Big(\int_Y \scalarproduct{E_0(x, y)}{ E^k(y)} \d y \Big) \theta(x)\d x \nonumber\\
  &= - \i \omega \eps_0 \sum_{l \in \NE} \int_{\Omega}(\eps_{\textrm{eff}})_{kl} E_l(x) \theta(x) \d x \nonumber \\
  &= \int_{\Omega} \scalarproduct[\big]{- \i \omega \eps_0  \eps_{\textrm{eff}} \hat{E}(x)}{ \theta(x)\e_k} \d x\, . \label{eq:proof of effective equations - 4}
\end{align}
As $\theta \in C_c^{\infty}(\Omega; \R)$ was chosen arbitrarily,~\eqref{eq:proof of effective equations - 3} and~\eqref{eq:proof of effective equations - 4} imply  that
\begin{equation*}
  (\curl \hat{H})_k = - \i \omega \eps_0 (\eps_{\textrm{eff}} \hat{E})_k \quad \textrm{ for all } k \in \NE\, .
\end{equation*}
This shows~\eqref{eq:effective equations 2 - derivation of effective
  equations} and hence the theorem is proved.
\end{proof}

\begin{remark}[Well-posedness of system~\eqref{eq:effective equations - derivation of effective equations}]
  We claim that the effective Maxwell system \eqref{eq:effective
    equations - derivation of effective equations} forms a complete
  set of equations. To be more precise, we show the following: Let $R
  \subset \Omega$ be a cuboid that is parallel to the axes and let
  $\Omega$ be a bounded
  Lipschitz domain. Assume that $\hat{\mu}$
  is real and positive definite on
  $V \coloneqq \spann \set{\e_k \given k \in \LH}$---that is, there is
  a constant $\alpha > 0$ such that
  $\scalarproduct{\hat{\mu} \xi}{\xi} \geq \alpha \abs{\xi}^2$ for all
  $\xi \in V$. We also assume that $\hat{\eps}$ is real and positive
  definite on $\spann \set{\e_k \given k \in \NE}$, and that
  $\mu_0 >0$, $\Re \eps_0 > 0$, and $\Im \eps_0 >0$. Let
  $(\hat{E}, \hat{H})$ be a solution to~\eqref{eq:effective equations
    - derivation of effective equations} with boundary condition
  $\hat{H} \wedge \nu = 0$ on $\del \Omega$, and assume that both
  $\hat{E}$ and $\hat{H}$ are of class $H^1$ in $R$ and in
  $\Omega \setminus R$.  We claim that $\hat{E}$ and $\hat{H}$ are
  trivial.

To prove that $\hat{E} = \hat{H} = 0$, we first show that the integration by parts
formula 
\begin{equation}\label{eq:remark step 2}
\int_{\Omega} \scalarproduct[\big]{\curl \hat{H}}{\hat{E}} = \int_{\Omega} \scalarproduct[\big]{\hat{H}}{\curl \hat{E}}\, .
\end{equation}
holds.  Equation~\eqref{eq:remark step 2} is a consequence of an
integration by parts provided that the integral
$\int_{\del R} \jump{\hat{H}} (\hat{E} \wedge \nu)$ vanishes, where
$\jump{\hat{H}}$ is the jump of the field $\hat{H}$ across the
boundary $\del R$ and $\nu$ is the outward unit normal vector on
$\del R$. As $\hat{E} \in H(\curl, \Omega)$, there holds
$\jump{\hat{E} \wedge \nu} = 0$ (i.e., the tangential components of $\hat{E}$ do
not jump). 

Let $\Gamma$ be one face of $R$. To prove
$\int_{\Gamma} \jump{\hat{H}} (\hat{E} \wedge \nu) = 0$, it suffices
to show that for all $k,l \in \set{1,2,3}$ with $k \neq l$ and
$\scalarproduct{\e_k}{\nu} = \scalarproduct{\e_l}{\nu} = 0$, we have
that $\hat{E}_k \jump{\hat{H_l}}_{\Gamma} = 0$. We obtain this
relation from~\eqref{eq:effective equations 2 - derivation of
  effective equations} and \eqref{eq:effective equations 4 -
  derivation of effective equations}: Indeed, for $k \notin \NE$,
there holds $\hat{E}_k = 0$ because of~\eqref{eq:effective equations 4
  - derivation of effective equations}. On the other hand,
by~\eqref{eq:effective equations 2 - derivation of effective
  equations}, for $k \in \NE$ there holds
$\del_m H_l - \del_l H_m = \mp \i \omega \eps_0 (\hat{\eps}
\hat{E})_k$ in the distributional sense, where $\nu = \e_m$. This
implies that $\jump{H_l}_{\Gamma} = 0$.

We now show that for almost all $x \in \Omega$
\begin{equation}\label{eq:remark step 1}
  \scalarproduct[\big]{\curl \hat{H}(x)}{\hat{E}(x)} = - \i \omega \eps_0 \scalarproduct[\big]{\hat{\eps} \hat{E}(x)}{\hat{E}(x)}\, .
\end{equation}
In $\Omega \setminus R$, this identity is a consequence
of~\eqref{eq:effective equations 3 - derivation of effective
  equations}. In $R$ on the other hand, we conclude
from~\eqref{eq:effective equations 4 - derivation of effective
  equations} and~\eqref{eq:effective equations 2 - derivation of
  effective equations} that, for $x \in R$,
\begin{equation*}
  \scalarproduct[\big]{\curl \hat{H}(x)}{\hat{E}(x)} = \sum_{k \in \NE} ( \curl \hat{H}(x) )_k \,  \overline{\hat{E}_k(x)} = - \i \omega \eps_0 \sum_{k \in \NE}  (\eps_{\eff} \hat{E}(x))_k \, \overline{\hat{E}_k(x)}\, .
\end{equation*}
Applying again~\eqref{eq:effective equations 4 - derivation of effective equations}, we obtain~\eqref{eq:remark step 1}.

From~\eqref{eq:effective equations 1 - derivation of effective
  equations}, the integration by parts formula~\eqref{eq:remark step
  2}, and~\eqref{eq:remark step 1}, we obtain
\begin{equation*}
  \int_{\Omega} \scalarproduct[\big]{\hat{\mu} \hat{H}}{\hat{H}} =
  -\frac{\i}{\omega \mu_0} \int_{\Omega} \scalarproduct[\big]{\curl
    \hat{E}}{\hat{H}} = -\frac{\i}{\omega \mu_0} \int_{\Omega}
  \scalarproduct[\big]{\hat{E}}{\curl \hat{H}}
  = \frac{\eps_0}{\mu_0} \scalarproduct[\big]{\hat{\eps}\hat{E}}{\hat{E}}\, .
\end{equation*}
As $\hat{\mu}$ and $\mu_0$ are assumed to be real, by taking the imaginary part, we
find that
$\int_{\Omega} \Im \eps_0 \scalarproduct{\hat{\eps} \hat{E}}{\hat{E}}
= 0$, and hence $\scalarproduct{\hat{\eps} \hat{E}}{\hat{E}} = 0$
almost everywhere in $\Omega$. This implies that $\hat{E} = 0$ in
$\Omega \setminus R$ and, taking into account that $\hat{\eps}$ is
positive definite on $\spann \set{\e_k \given k \in \NE}$, we also
find $\hat{E} = 0$ in $R$. We can therefore conclude
from~\eqref{eq:effective equations 1 - derivation of effective
  equations} that $\hat{H} $ vanishes in $\Omega \setminus R$ and that
$\hat{H}_k = 0$ in $R$ for all $k \in \LH$. On account
of~\eqref{eq:effective equations 5 - derivation of effective
  equations}, $\hat{H} = 0$ in $\Omega$. This shows the uniqueness of
solutions.
\end{remark}

\section{Discussion of examples} \label{section:examples}

In this section, we apply Theorem~\ref{theorem:macroscopic equations} to some examples.
In what follows, $d_E$ and $d_H$ denote the dimension of
the solution space to  the cell problem~\eqref{eq:cell problem E} of
$E_0$ and \eqref{eq:cell problem H_0} of $H_0$,
respectively.  Due to Propositions \ref{proposition: dimension of
  solution space to the cell problem of E} and~\ref{proposition:for
  every element of the set L E there is a solution to the cell
  problem}, we find that $d_E = \abs{\NE}$ and $d_H=\abs{\LH}$.

\subsection{The metal ball}\label{example:metal ball}

 To define the metallic ball structure, we fix a number $r \in (0, 1/2)$, and set
\begin{equation}\label{eq:definition metal ball}
  \Sigma \coloneqq \set{y = (y_1, y_2, y_3) \in Y \given y_1^2 + y_2^2 +y_3^2 < r^2 }\, .
\end{equation}
A sketch of the periodicity cell is given in Figure~\ref{fig:the metal
  ball}. We note that $Y \setminus \overline{\Sigma}$ is a simple
Helmholtz domain. As each two opposite faces of $Y$ can be connected
by a loop in $Y
\setminus \overline{\Sigma}$, we have that $\LH = \set{1,2,3}$. On the
other hand, we find no loop in $\Sigma$ that connects two opposite
faces of $Y$; hence $\NE = \set{1,2,3}$. To summarise, we have that
\begin{center}
  \begin{tabular}{ c | c | c | c   }
    $\NE$  & $d_E$ & $\LH$ & $d_H$\\ \hline
    $\set{1,2,3}$       & $3$    & $\set{1,2,3}$      & $3$ \\
  \end{tabular}\, .
\end{center}
The Maxwell system~\eqref{eq:effective equations - derivation of effective equations} is of the usual form; to be more precise, the following result holds.
\begin{corollary}[Macroscopic equations of the metal ball]
  For $\Sigma$ as in~\eqref{eq:definition metal ball}, a sequence
  $(\Eeta, \Heta)_{\eta}$, and limits $\hat{E}$ and $\hat{H}$ as in Theorem~\ref{theorem:macroscopic equations},  the macroscopic equations read
 \begin{subequations}\label{eq:effective equations for metal ball}
  \begin{align}
    \label{eq:effective equations for metal ball - 1}
   \curl \hat{E} &= \phantom{-} \i \omega \mu_0 \hat{\mu} \hat{H} \quad \textrm{ in } \Omega\, ,\\
    \label{eq:effective equations for metal ball - 2}
   \curl \hat{H}&= - \i \omega \eps_0 \hat{\eps}\hat{E} \: \,\quad \textrm{ in } \Omega \, .    
  \end{align}
  \end{subequations}
\end{corollary}

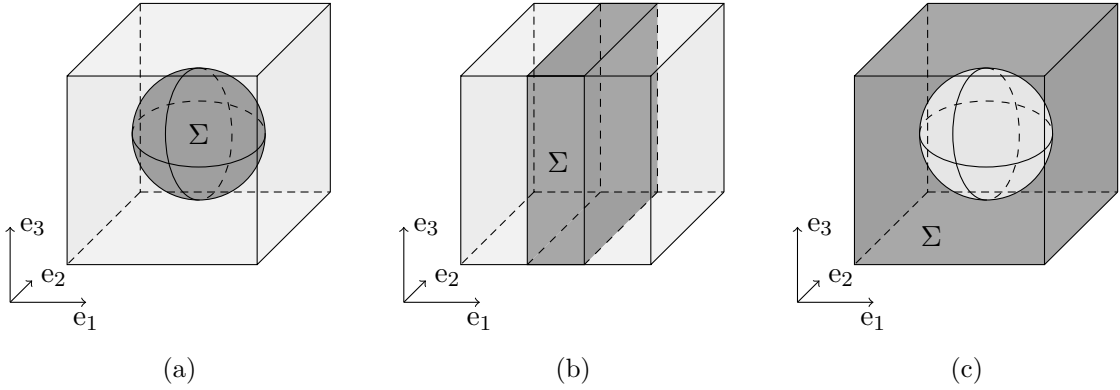
\begin{figure}
  \centering
  \begin{subfigure}[t]{0.3\textwidth}
    \begin{tikzpicture}[scale=2.5]
  \coordinate (A) at (-0.5, -0.5, 0.5);
  \coordinate (B) at (0.5,-0.5,0.5);
  \coordinate (C) at (0.5,0.5,0.5);
  \coordinate (D) at (-0.5, 0.5, 0.5);

  \coordinate (E) at (-0.5, 0.5, -0.5);
  \coordinate (F) at (0.5, 0.5, -0.5);
  \coordinate (G) at (0.5, -0.5, -0.5);
  \coordinate (H) at (-0.5, -0.5, -0.5);

  \fill[gray!20!white, opacity=.5] (A) -- (B) -- (C) -- (D) -- cycle;
  \fill[gray!20!white, opacity=.5] (E) -- (F) -- (G) --(H) -- cycle;
  \fill[gray!20!white, opacity=.5] (D) -- (E) -- (H) -- (A)-- cycle;
  \fill[gray!20!white, opacity=.5] (B) -- (C) -- (F) -- (G) -- cycle;
  
  \fill[gray!75!white] (0,0) circle (.35cm);
  
  \draw[] (A) -- (B) -- (C) -- (D) --cycle (E) -- (D) (F) -- (C) (G) -- (B);
  \draw[] (E) -- (F) -- (G) ;
  \draw[densely dashed] (E) -- (H) (H) -- (G) (H) -- (A);

  \draw[] (0,0) circle (.35cm);
  \draw[] (-.35,0) arc (180:360:.35cm and .175cm);        
  \draw[dashed] (-.35,0) arc (180:0:.35cm and .175cm);
  \draw[] (0,.35) arc (90:270:0.175cm and 0.35cm);
  \draw[dashed] (0,.35) arc (90:-90:0.175cm and 0.35cm);

  \draw[->] ( -.8, -.7, .5)--( -.4, -.7, .5 );
  \draw[->] (-.8, -.7, .5) -- (-.8, -.3, .5);
  \draw[->] (-.8, -.7, .5) -- (-.8, -.7, .2);

  \node[] at (0,0,0) {$\Sigma$};
  \node[] at (-.4, -.8, .5){$\e_1$};
  \node[] at (-.68, -.3, .5){$\e_3$};
  \node[] at (-.8, -.8, .-.1){$\e_2$};
  \end{tikzpicture}
   \caption{ }
   \label{fig:the metal ball}
  \end{subfigure}
\hfill
\begin{subfigure}[t]{0.3\textwidth}
  \begin{tikzpicture}[scale=2.5]
  \coordinate (A) at (-0.5, -0.5, 0.5);
  \coordinate (B) at (0.5,-0.5,0.5);
  \coordinate (C) at (0.5,0.5,0.5);
  \coordinate (D) at (-0.5, 0.5, 0.5);

  \coordinate (E) at (-0.5, 0.5, -0.5);
  \coordinate (F) at (0.5, 0.5, -0.5);
  \coordinate (G) at (0.5, -0.5, -0.5);
  \coordinate (H) at (-0.5, -0.5, -0.5);

  \coordinate (M1) at (-0.15, -0.5, 0.5);
  \coordinate (M2) at (0.15, -0.5, 0.5);
  \coordinate (M3) at (0.15, 0.5, 0.5);
  \coordinate (M4) at (-0.15, 0.5, 0.5);
  \coordinate (M5) at (-0.15, 0.5, -0.5);
  \coordinate (M6) at (0.15,0.5,-0.5);
  \coordinate (M7) at (0.15, -0.5, -0.5);
  \coordinate (M8) at (-0.15, -0.5, -0.5);

  \fill[gray!20!white, opacity=.5] (A) -- (B) -- (C) -- (D) -- cycle;
  \fill[gray!20!white, opacity=.5] (E) -- (F) -- (G) --(H) -- cycle;
  \fill[gray!20!white, opacity=.5] (D) -- (E) -- (H) -- (A)-- cycle;
  \fill[gray!20!white, opacity=.5] (B) -- (C) -- (F) -- (G) -- cycle;
  
  \fill[gray!75!white, opacity=.9] (M4) -- (M5) -- (M6) -- (M3) -- cycle;
  \fill[gray!75!white, opacity=.9] (M1) -- (M4) -- (M5) -- (M8) -- cycle;
  \fill[gray!75!white, opacity=.9] (M2) -- (M3) -- (M6) -- (M7) -- cycle;
  \fill[gray!75!white, opacity=.9] (M1) -- (M2) -- (M7) -- (M8) -- cycle;
  
  \draw[] (A) -- (B) -- (C) -- (D) --cycle (E) -- (D) (F) -- (C) (G) -- (B);
  \draw[] (E) -- (F) -- (G) ;
  \draw[densely dashed] (E) -- (H) (H) -- (G) (H) -- (A);

  \draw[] (M1) -- (M2) -- (M3) -- (M4) -- cycle;
  \draw[] (M4) -- (M5) -- (M6) -- (M3) -- cycle;
  \draw[dashed] (M5) -- (M8) (M7) -- (M6); 
  \draw[dashed] (M1) -- (M8)  (M7) -- (M2);
  
  \draw[->] ( -.8, -.7, .5)--( -.4, -.7, .5 );
  \draw[->] (-.8, -.7, .5) -- (-.8, -.3, .5);
  \draw[->] (-.8, -.7, .5) -- (-.8, -.7, .2);

  \node[] at (-0.18, -.15, 0) {$\Sigma$};
  \node[] at (-.4, -.8, .5){$\e_1$};
  \node[] at (-.68, -.3, .5){$\e_3$};
  \node[] at (-.8, -.8, .-.1){$\e_2$};
\end{tikzpicture}
  \caption{ }
  \label{fig:the metal plate}
\end{subfigure}
\hfill
\begin{subfigure}[t]{0.3\textwidth}
  \begin{tikzpicture}[scale=2.5]
  \coordinate (A) at (-0.5, -0.5, 0.5);
  \coordinate (B) at (0.5,-0.5,0.5);
  \coordinate (C) at (0.5,0.5,0.5);
  \coordinate (D) at (-0.5, 0.5, 0.5);

  \coordinate (E) at (-0.5, 0.5, -0.5);
  \coordinate (F) at (0.5, 0.5, -0.5);
  \coordinate (G) at (0.5, -0.5, -0.5);
  \coordinate (H) at (-0.5, -0.5, -0.5);

  \fill[gray!75!white, opacity=.9] (A) -- (B) -- (C) -- (D) -- cycle;
  \fill[gray!75!white, opacity=.9] (E) -- (F) -- (G) --(H) -- cycle;
  \fill[gray!75!white, opacity=.9] (D) -- (E) -- (H) -- (A)-- cycle;
  \fill[gray!75!white, opacity=.9] (B) -- (C) -- (F) -- (G) -- cycle;
  
  \fill[gray!20!white] (0,0) circle (.35cm);

  \draw[] (A) -- (B) -- (C) -- (D) --cycle (E) -- (D) (F) -- (C) (G) -- (B);
  \draw[] (E) -- (F) -- (G) ;
  \draw[densely dashed] (E) -- (H) (H) -- (G) (H) -- (A);

  \draw[] (0,0) circle (.35cm);
  \draw[] (-.35,0) arc (180:360:.35cm and .175cm);        
  \draw[dashed] (-.35,0) arc (180:0:.35cm and .175cm);
  \draw[] (0,.35) arc (90:270:0.175cm and 0.35cm);
  \draw[dashed] (0,.35) arc (90:-90:0.175cm and 0.35cm);

  \draw[->] ( -.8, -.7, .5)--( -.4, -.7, .5 );
  \draw[->] (-.8, -.7, .5) -- (-.8, -.3, .5);
  \draw[->] (-.8, -.7, .5) -- (-.8, -.7, .2);
  
  \node[] at (-.25, -0.5, .1) {$\Sigma$};
  \node[] at (-.4, -.8, .5){$\e_1$};
  \node[] at (-.68, -.3, .5){$\e_3$};
  \node[] at (-.8, -.8, .-.1){$\e_2$};
\end{tikzpicture}
  \caption{ }
  \label{fig:the air ball}
\end{subfigure}
\caption{The periodicity cell $Y$ is represented by the cube. (a)  The
  metal ball of Example~\ref{example:metal ball}. (b) The metal plate
  of Example~\ref{example:metal plate}. (c) A sketch of the air ball
  (see Remark~\ref{remark:effective equations of air ball}).}
\end{figure}

\subsection{The metal plate}\label{example:metal plate}

To define the metal plate structure, fix a number $\gamma \in (0, 1/2)$ and set
\begin{equation}\label{eq:definition metal plate}
  \Sigma \coloneqq \set{y = (y_1, y_2, y_3) \in Y \given y_1 \in (- \gamma, \gamma) }\, .
\end{equation}
We refer to Figure~\ref{fig:the metal plate} for a sketch of the periodicity cell $Y$. Observe that $Y \setminus \overline{\Sigma}$ is a simple Helmholtz domain. We obtain the table
\begin{center}
  \begin{tabular}{ c | c | c | c   }
    $\NE$  & $d_E$ & $\LH$ & $d_H$\\ \hline
    $\set{1}$             & $1$    & $\set{2,3}$         & $2$ \\
  \end{tabular}\, .
\end{center}
In fact, we do not only know the dimensions of the solution spaces
to~\eqref{eq:cell problem E} and \eqref{eq:cell problem H_0} but also
bases for these spaces. Indeed, for the volume fraction
$\alpha \coloneqq \abs{Y \setminus \overline{\Sigma}}$, the field
$E^1 \colon Y \to \C^3$ given by
$E^1(y) \coloneqq \e_1 \alpha^{-1} \indicator{Y \setminus
  \overline{\Sigma}}(y)$ is a solution to~\eqref{eq:cell problem E}
with $\meanint_Y E^1 = \e_1$. On the other hand, for
$k \in \set{2,3}$, the field $H^k \colon Y \to \C^3$,
$H^k(y) \coloneqq \e_k \indicator{Y \setminus \overline{\Sigma}}(y)$
is a solution to~\eqref{eq:cell problem H_0}. By the first part of the definition
of the geometric average, $\oint H^k = \e_k$ and hence
$\set{H^2, H^3}$ is a basis of the solution space to~\eqref{eq:cell
  problem H_0}. Having the basis for the solution spaces at hand, we
can compute $\eps_{\eff}$ and $\mu_{\eff}$ defined in
\eqref{eq:definition eps eff} and \eqref{eq:definition mu eff}: we have that
$\eps_{\eff} = \alpha^{-1} \diag(1, 0, 0)$ and
$\mu_{\eff} = \alpha \, \diag(0, 1, 1)$.  An application of
Theorem~\ref{theorem:macroscopic equations} yields the effective
equations for the metal plate.

\begin{corollary}[Macroscopic equations for the metal plate]
  Let $\Sigma$ be as in~\eqref{eq:definition metal plate} and let $\alpha \coloneqq \abs{Y \setminus
  \overline{\Sigma}}$. For a sequence
  $(\Eeta, \Heta)_{\eta}$, and limit fields $\hat{E}$ and $\hat{H}$ as
  in Theorem~\ref{theorem:macroscopic equations}, the macroscopic equations read
  \begin{subequations}\label{eq:effective equations for metal plate}
  \begin{align}
    \label{eq:effective equations for metal plate - 1}
   \curl \hat{E} &= \phantom{-} \i \omega \mu_0 \hat{H} \phantom{-}\qquad \: \: \;\textrm{ in } \Omega \setminus R\, ,\\
    \label{eq:effective equations for metal plate - 2}
   \curl \hat{H}&= - \i \omega \eps_0 \hat{E} \: \qquad
                  \:\:\;\phantom{-} \textrm{ in } \Omega \setminus R\,
                  ,  \\        
    \label{eq:effective equations for metal plate - 6}
    (\del_3, -\del_2 ) \hat{E}_1  &=\phantom{-} \i \omega \mu_0\alpha (\hat{H}_2, \hat{H}_3) \:  \,\textrm{ in } R \, ,\\
    \label{eq:effective equations for metal plate - 3}
    \del_2 \hat{H}_3- \del_3 \hat{H}_2 &= - \i \omega \eps_0 \alpha^{-1} \hat{E}_1\phantom{-}\quad \, \textrm{ in } R \, ,\\
    \label{eq:effective equations for metal plate - 4}
    \hat{E}_2=\hat{E}_3 &= 0 \qquad \qquad \quad \: \: \: \: \:  \;\phantom{-}\textrm{ in } R\, , \\
    \label{eq:effective equations for metal plate - 5}
                      \hat{H}_1&= 0 \qquad \qquad \quad \:\:\:\: \: \;\phantom{-} \textrm{ in } R \, .
  \end{align}
  \end{subequations}
\end{corollary}

\subsection{The air cylinder}\label{example:air cylinder}

To define the metallic box with a cylinder removed, we fix a number $r \in (0, 1/2)$ and set 
\begin{equation*}
  \Sigma \coloneqq Y \setminus \set{y = (y_1, y_2, y_3) \in Y \given y_1^2 + y_2^2 < r^2 }\, .
\end{equation*}
A sketch of the periodicity cell is given in Figure~\ref{fig:the air cylinder}. The air cylinder $Y \setminus \overline{\Sigma}$ is a simple Helmholtz domain. We obtain
\begin{center}
  \begin{tabular}{ c | c | c | c   }
    $\NE$  & $d_E$ & $\LH$ & $d_H$\\ \hline
    $\emptyset$       & $0$    & $\set{3}$            & $1$ \\
  \end{tabular}\, .
\end{center}
Once again, we do not only know the dimension of the solution space to~\eqref{eq:cell problem H_0} but also its basis. Indeed, the field $H^3 \colon Y \to \C^3$ given by $H^3(y) \coloneqq \e_k \indicator{Y \setminus \overline{\Sigma}}(y)$ is a solution to~\eqref{eq:cell problem H_0}. We can thus compute $\eps_{\eff}$ and $\mu_{\eff}$; for $\alpha \coloneqq \abs{Y \setminus \overline{\Sigma}}$, we find that $\eps_{\eff} = 0$ and $\mu_{\eff} =\alpha \, \diag(0, 0, 1)$.
Although the solution space to the cell problem of $H_0$ is not
trivial, there is only the trivial solution to the effective equations
in $R$; that is,
\begin{equation*}
  \hat{E} = \hat{H}= 0 \quad \textrm{ in } R\, .
\end{equation*}
 Note that $\hat{H}_3= 0$ is a consequence
of~\eqref{eq:MaxSeq1-intro} and $\hat{E} = 0$ in $R$.

\begin{remark}\label{remark:effective equations of air ball}
  Instead of the air cylinder, we can also consider the air ball (see
  Figure~\ref{fig:the air ball} for a sketch) and find that there are
  also only the trivial solutions $\hat{E} = \hat{H}= 0$ in $R$.
\end{remark}

\begin{figure}
  \centering
  \begin{subfigure}[t]{0.3\textwidth}
    \begin{tikzpicture}[scale=2.5]
  \coordinate (A) at (-0.5, -0.5, 0.5);
  \coordinate (B) at (0.5,-0.5,0.5);
  \coordinate (C) at (0.5,0.5,0.5);
  \coordinate (D) at (-0.5, 0.5, 0.5);

  \coordinate (E) at (-0.5, 0.5, -0.5);
  \coordinate (F) at (0.5, 0.5, -0.5);
  \coordinate (G) at (0.5, -0.5, -0.5);
  \coordinate (H) at (-0.5, -0.5, -0.5);

  \fill[gray!75!white, opacity=.9] (A) -- (B) -- (C) -- (D) -- cycle;
  \fill[gray!75!white, opacity=.9] (E) -- (F) -- (G) --(H) -- cycle;
  \fill[gray!75!white, opacity=.9] (D) -- (E) -- (H) -- (A)-- cycle;
  \fill[gray!75!white, opacity=.9] (B) -- (C) -- (F) -- (G) -- cycle;
  
  \fill[gray!20!white] (-.25, -.5) -- (-.25,.5) 
                                arc(180:0:0.25cm and 0.125cm) -- (0.25, -0.5) 
                                (-0.25, -.5) arc(180:360:0.25cm and 0.125cm);
  \filldraw[gray!20!white] (-.25,-.5) -- (0.25,-.5);
  
  \draw[] (A) -- (B) -- (C) -- (D) --cycle (E) -- (D) (F) -- (C) (G) -- (B);
  \draw[] (E) -- (F) -- (G) ;
  \draw[densely dashed] (E) -- (H) (H) -- (G) (H) -- (A);

  \draw[] (-.25,0.5) arc (180:-180:0.25cm and 0.125cm);
  \draw[dashed] (-.25,-.5) arc (180:360:0.25cm and 0.125cm);
  \draw[dashed] (-.25,-.5) arc (180:0:0.25cm and 0.125cm);
  \draw[dashed] (-0.25, -.5) -- (-0.25, .5);
  \draw[dashed] (.25,-.5) -- (.25,.5);

  \draw[->] ( -.8, -.7, .5)--( -.4, -.7, .5 );
  \draw[->] (-.8, -.7, .5) -- (-.8, -.3, .5);
  \draw[->] (-.8, -.7, .5) -- (-.8, -.7, .2);

  \node[] at (-.25, .1, .55) {$\Sigma$};
  \node[] at (-.4, -.8, .5){$\e_1$};
  \node[] at (-.68, -.3, .5){$\e_3$};
  \node[] at (-.8, -.8, .-.1){$\e_2$};
\end{tikzpicture}
  \caption{ }
  \label{fig:the air cylinder}
  \end{subfigure}
\hfill
\begin{subfigure}[t]{0.3\textwidth}
  \begin{tikzpicture}[scale=2.5]
  \coordinate (A) at (-0.5, -0.5, 0.5);
  \coordinate (B) at (0.5,-0.5,0.5);
  \coordinate (C) at (0.5,0.5,0.5);
  \coordinate (D) at (-0.5, 0.5, 0.5);

  \coordinate (E) at (-0.5, 0.5, -0.5);
  \coordinate (F) at (0.5, 0.5, -0.5);
  \coordinate (G) at (0.5, -0.5, -0.5);
  \coordinate (H) at (-0.5, -0.5, -0.5);

  \fill[gray!20!white, opacity=.5] (A) -- (B) -- (C) -- (D) -- cycle;
  \fill[gray!20!white, opacity=.5] (E) -- (F) -- (G) --(H) -- cycle;
  \fill[gray!20!white, opacity=.5] (D) -- (E) -- (H) -- (A)-- cycle;
  \fill[gray!20!white, opacity=.5] (B) -- (C) -- (F) -- (G) -- cycle;
  
  \fill[gray!75!white] (-.25, -.5) -- (-.25,.5) 
                                arc(180:0:0.25cm and 0.125cm) -- (0.25, -0.5) 
                                (-0.25, -.5) arc(180:360:0.25cm and 0.125cm);
  \filldraw[gray!75!white] (-.25,-.5) -- (0.25,-.5);
  
  \draw[] (A) -- (B) -- (C) -- (D) --cycle (E) -- (D) (F) -- (C) (G) -- (B);
  \draw[] (E) -- (F) -- (G) ;
  \draw[densely dashed] (E) -- (H) (H) -- (G) (H) -- (A);

  \draw[] (-.25,0.5) arc (180:-180:0.25cm and 0.125cm);
  \draw[dashed] (-.25,-.5) arc (180:360:0.25cm and 0.125cm);
  \draw[dashed] (-.25,-.5) arc (180:0:0.25cm and 0.125cm);
  \draw[dashed] (-0.25, -.5) -- (-0.25, .5);
  \draw[dashed] (.25,-.5) -- (.25,.5);

  \draw[->] ( -.8, -.7, .5)--( -.4, -.7, .5 );
  \draw[->] (-.8, -.7, .5) -- (-.8, -.3, .5);
  \draw[->] (-.8, -.7, .5) -- (-.8, -.7, .2);

  \node[] at (0,0,0) {$\Sigma$};
\node[] at (-.4, -.8, .5){$\e_1$};
  \node[] at (-.68, -.3, .5){$\e_3$};
  \node[] at (-.8, -.8, .-.1){$\e_2$};
  
\end{tikzpicture}
  \caption{ }
  \label{fig:the metal cylinder}
\end{subfigure}
\hfill
\begin{subfigure}[t]{0.3\textwidth}
  \begin{tikzpicture}[scale=2.5]
  \coordinate (A) at (-0.5, -0.5, 0.5);
  \coordinate (B) at (0.5,-0.5,0.5);
  \coordinate (C) at (0.5,0.5,0.5);
  \coordinate (D) at (-0.5, 0.5, 0.5);

  \coordinate (E) at (-0.5, 0.5, -0.5);
  \coordinate (F) at (0.5, 0.5, -0.5);
  \coordinate (G) at (0.5, -0.5, -0.5);
  \coordinate (H) at (-0.5, -0.5, -0.5);

  \fill[gray!20!white, opacity=.5] (A) -- (B) -- (C) -- (D) -- cycle;
  \fill[gray!20!white, opacity=.5] (E) -- (F) -- (G) --(H) -- cycle;
  \fill[gray!20!white, opacity=.5] (D) -- (E) -- (H) -- (A)-- cycle;
  \fill[gray!20!white, opacity=.5] (B) -- (C) -- (F) -- (G) -- cycle;
  
  \draw[rotate=0] (0,0) ellipse (15pt and 7.5pt);
  \fill[gray!75!white] (0,0) ellipse (15pt and 7.5pt);
  \draw[black] (-.2,0,0) to[bend left] (.2,0,0);
  \fill[gray!20!white]  (-.2,0,0) to[bend left] (.2,0,0);
  \fill[gray!20!white] (-.2,0, 0) to[out=346, in=198] (.2,0,0);
  \draw[black] (-.35,.07) to[bend right] (.35,.07);
  
    \draw[] (A) -- (B) -- (C) -- (D) --cycle (E) -- (D) (F) -- (C) (G) -- (B);
  \draw[] (E) -- (F) -- (G) ;
  \draw[densely dashed] (E) -- (H) (H) -- (G) (H) -- (A);

  \draw[->] ( -.8, -.7, .5)--( -.4, -.7, .5 );
  \draw[->] (-.8, -.7, .5) -- (-.8, -.3, .5);
  \draw[->] (-.8, -.7, .5) -- (-.8, -.7, .2);
  
  \node[] at (-0.15, -.13, 0) {$\Sigma$};
  \node[] at (-.4, -.8, .5){$\e_1$};
  \node[] at (-.68, -.3, .5){$\e_3$};
  \node[] at (-.8, -.8, .-.1){$\e_2$};
  
\end{tikzpicture}
  \caption{ }
  \label{fig:torus as metal part}
\end{subfigure}
\caption{The periodicity cell $Y$ is represented by the cube. (a) The
  air cylinder $Y \setminus \overline{\Sigma}$ of Example~\ref{example:air
  cylinder} (b) The metal cylinder $\Sigma$ of Example~\ref{example:metal cylinder}. (c) $\Sigma$ is a $2$-dimensional full torus. In this case, $Y \setminus \overline{\Sigma}$ is not a simple Helmholtz domain.}
\end{figure}
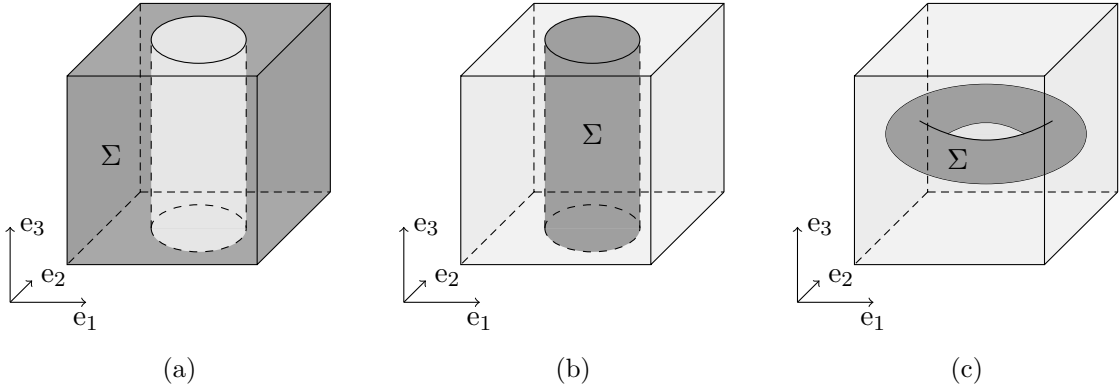
\subsection{The metal cylinder}\label{example:metal cylinder}

Fix a number $r \in (0, 1/2)$. To model a metallic cylinder,  $\Sigma$ is defined as the set
\begin{equation}\label{eq:definition metal cylinder}
  \Sigma \coloneqq \set{y = (y_1, y_2, y_3) \in Y \given y_1^2 + y_2^2 < r^2 }\, .
\end{equation}
A sketch of the periodicity cell is given in Figure~\ref{fig:the metal cylinder}. 

We claim that $Y \setminus \overline{\Sigma}$ is a simple Helmholtz domain.
Indeed, there are only the three standard nontrivial loops in $Y
\setminus \overline{\Sigma}$; namely, $\gamma_1, \gamma_2$, and
$\gamma_3$, which are given by $\gamma_1(t) \coloneqq (t - 1/2, -1/2,
-1/2)$, $\gamma_{2}(t) \coloneqq (-1/2, t-1/2, -1/2)$, and
$\gamma_3(t) \coloneqq (-1/2, -1/2, t-1/2)$ for $t \in [0, 1]$. Thus, every $L_{\sharp}^2(Y; \C^3)$-vector field $u$ that is curl free in $Y \setminus \overline{\Sigma}$ can be written as $u = \nabla \Theta + c_0$ for $\Theta \in H_{\sharp}^1(Y \setminus \overline{\Sigma})$ and $c_0 \in \C^3$. 

We find the table
\begin{center}
  \begin{tabular}{ c | c | c | c   }
    $\NE$  & $d_E$ & $\LH$ & $d_H$\\ \hline
    $\set{1,2}$          & $2$    & $\set{1,2,3}$            & $3$ \\
  \end{tabular}\, .
\end{center}
As for the metal plate, we find an interesting non-trivial limit system.
\begin{corollary}[Macroscopic equations for the metal cylinder]
  For $\Sigma$ as in~\eqref{eq:definition metal cylinder}, a
  sequence $(\Eeta, \Heta)_{\eta}$, and limit fields $\hat{E}$ and $\hat{H}$ as in Theorem~\ref{theorem:macroscopic equations}, the macroscopic equations read
  \begin{subequations}\label{eq:effective equations for metal cylinder}
  \begin{align}
    \label{eq:effective equations for metal cylinder - 1}
    \curl \hat{E} &= \phantom{-} \i \omega \mu_0 \hat{\mu} \hat{H} \quad \:\:\:\, \textrm{ in } \Omega \, ,\\
    \label{eq:effective equations for metal cylinder - 2}
    \curl \hat{H}&= - \i \omega \eps_0 \hat{E} \: \qquad \:\: \textrm{ in } \Omega \setminus R\, ,  \\              
 \label{eq:effective equations for metal cylinder - 3}
    \del_2 \hat{H}_3- \del_3 \hat{H}_2 &= - \i \omega \eps_0 (\hat{\eps} \hat{E})_1 \quad  \textrm{ in } \Omega \, ,\\
    \label{eq:effective equations for metal cylinder - 4}
           \del_3 \hat{H}_1 - \del_1 \hat{H}_3 &= - \i \omega \eps_0 (\hat{\eps} \hat{E})_2\quad \textrm{ in } \Omega\, , \\
    \label{eq:effective equations for metal cylinder - 5}
    \hat{E}_3 &= 0 \qquad \qquad \quad \:\:\: \,  \textrm{ in } R\, .
  \end{align}
  \end{subequations}
\end{corollary}

\begin{proof}
  We have that $\LH = \set{1,2,3}$, $\NE = \set{1,2}$, and $\hat{\eps} = \id_3$ in $\Omega \setminus R$. Thus,~\eqref{eq:effective equations for metal cylinder - 1} and~\eqref{eq:effective equations for metal cylinder - 2} follow from~\eqref{eq:effective equations 1 - derivation of effective equations} and~\eqref{eq:effective equations 3 - derivation of effective equations}, respectively. Equations~\eqref{eq:effective equations for metal cylinder - 3} and~\eqref{eq:effective equations for metal cylinder - 4} follow from~\eqref{eq:effective equations 2 - derivation of effective equations}, and~\eqref{eq:effective equations for metal cylinder - 5} is a consequence of~\eqref{eq:effective equations 4 - derivation of effective equations}.
\end{proof}

\section*{Funding}
Support of both authors by DFG grant Schw 639/6-1 is gratefully
acknowledged.

\bibliographystyle{abbrv} 
\bibliography{lit-maixwell-2}

\end{document}